\documentclass{amsart}
\usepackage{amsmath, amssymb,amscd}

\input xy
\xyoption{all}

\newtheorem{theorem}{Theorem}[section]
\newtheorem{lemma}[theorem]{Lemma}
\newtheorem{corollary}[theorem]{Corollary}
\newtheorem{fact}[theorem]{Fact}
\newtheorem{proposition}[theorem]{Proposition}
\newtheorem{claim}[theorem]{Claim}

\theoremstyle{definition}
\newtheorem{example}[theorem]{Example}
\newtheorem{remark}[theorem]{Remark}
\newtheorem{definition}[theorem]{Definition}

\newtheorem{question}[theorem]{Question}

\def\hsd{\operatorname{HSD}}
\def\acl{\operatorname{acl}}
\def\dcl{\operatorname{dcl}}
\def\tp{\operatorname{tp}}
\def\cb{\operatorname{Cb}}
\def\spec{\operatorname{Spec}}
\def\sym{\operatorname{Sym}}

\def\id{\operatorname{id}}
\def\hom{\operatorname{Hom}}

\def\O{\mathcal{O}}

\def\E{\mathcal E}

\def\jet{\operatorname{Jet}}

\def\D{\mathcal{D}}
\def\alg{\operatorname{alg}}

\def\res{\operatorname{Res}}

\newcommand{\cD}{{\mathcal D}}
\newcommand{\CC}{{\mathbb C}}
\newcommand{\NN}{{\mathbb N}}
\newcommand{\ZZ}{{\mathbb Z}}
\newcommand{\QQ}{{\mathbb Q}}
\begin{document}

\title[Generalised Hasse-Schmidt varieties and their jet spaces]{Generalised Hasse-Schmidt varieties\\ and their jet spaces}

\author{Rahim Moosa}
\address{Rahim Moosa\\
University of Waterloo\\
Department of Pure Mathematics\\
200 University Avenue West\\
Waterloo, Ontario \  N2L 3G1\\
Canada}
\email{rmoosa@math.uwaterloo.ca}

\thanks{R. Moosa was supported by an NSERC Discovery Grant.}

\author{Thomas Scanlon}
\address{Thomas Scanlon\\
University of California, Berkeley\\
Department of Mathematics\\
Evans Hall\\
Berkeley, CA \ 94720-3480\\
USA}
\email{scanlon@math.berkeley.edu}

\thanks{T. Scanlon was partially supported by NSF Grant  CAREER
DMS-0450010, and a Templeton Infinity Grant.}

\date{November 1, 2010}

 
\begin{abstract}
Building on the abstract notion of prolongation developed in~\cite{paperA}, the theory of {\em iterative Hasse-Schmidt rings} and {\em schemes} is introduced, simultaneously generalising  difference and (Hasse-Schmidt) differential rings and schemes.
This work provides a unified formalism for studying difference and differential
algebraic geometry, as well as other related geometries.
As an application, Hasse-Schmidt jet spaces are constructed generally, allowing the development of the theory for arbitrary systems of algebraic partial difference/differential equations, where constructions by earlier authors applied only to the finite-dimensional case.
In particular, it is shown that under appropriate separability assumptions a Hasse-Schmidt variety is determined by its jet spaces at a point.
\end{abstract}

\maketitle

\tableofcontents
\newpage

\section{Introduction}

The algebraic theories of ordinary and partial differential equations, difference equations, Hasse-Schmidt differential equations, and mixed difference-differential equations bear many formal analogies and some of the theory may be developed uniformly under the rubric of equations over rings with fixed additional operators.  In this paper, a continuation of~\cite{paperA}, we propose a unified theory of rings with stacks of compatible operators, what we call generalised iterative Hasse-Schmidt rings, and then undertake a detailed study of the infinitesimal structure of Hasse-Schmidt varieties showing how to define jet spaces for these Hasse-Schmidt varieties and that the jet spaces determine the varieties under a separability hypothesis.

Before we consider Hasse-Schmidt rings in full generality, let us consider the special case of ordinary differential rings.  Here we have a commutative ring $R$ given together with a derivation $\partial:R \to R$.   At one level, to say that $\partial$ is a derivation is simply to say that $\partial$ is additive and satisfies the Leibniz rule.  On the other hand, we could say that the exponential map $R \to R[\epsilon]/(\epsilon^2)$ given by $x \mapsto x + \partial(x) \epsilon$ is a ring homomorphism.    When $R$ is a ${\mathbb Q}$-algebra, this truncated exponential map lifts to a ring homomorphism $R \to R[[\epsilon]]$ given by $x \mapsto \sum \frac{1}{n!} \partial^n(x) \epsilon^n$.  If we define $\partial_n(x) := \frac{1}{n!} \partial^n(x)$, then the exponential map takes the form $x \mapsto \sum \partial_n(x) \epsilon^n$.   Let us note that we have a formula relating composites of the $\partial_n$ operators with single applications.  Indeed, $\binom{n+m}{n} \partial_{n+m} = \frac{(n+m)!}{n! m!} \frac{1}{(n+m)!} \partial^{n+m} = \frac{1}{n!} \partial^n \circ \frac{1}{m!} \partial^m = \partial_n \circ \partial_m$.

From the defining equation for $\partial_n$, it is clear that it gives no more information than is already given by the first derivative $\partial$.  However, we could consider the general category of Hasse-Schmidt differential rings which are rings $R$ given together with a sequence of additive operators $\partial_n:R \to R$ for which the function $R \to R[[\epsilon]]$ given by $a \mapsto \sum_{n=0}^\infty \partial_n(a) \epsilon^n$ is a ring homomorphism, and the operators satisfy the rule that $\partial_0=\id$ and $\binom{n+m}{n} \partial_{n+m} = \partial_n \circ \partial_m$.  Dropping the hypothesis that $R$ is a ${\mathbb Q}$-algebra, one finds Hasse-Schmidt differential rings for which the higher operators need not be determined by the first derivative.  Indeed, the language of Hasse-Schmidt differential rings is the appropriate framework for studying differential equations in positive characteristic.  

As explained already by Matsumura (Section~27 of~\cite{matsumura} ), the iteration rule, $\binom{n+m}{n} \partial_{n+m} =  \partial_n \circ \partial_m$, may be expressed as a commuting diagram.
Let $(R,\langle \partial_i : i \in {\mathbb N} \rangle)$ be a Hasse-Schmidt differential ring.   That is to say, the map $E_\epsilon:R \to R[[\epsilon]]$ given by $x \mapsto \sum \partial_n(x) \epsilon^n$ is a ring homomorphism.  Extending each $\partial_n$ continuously to $R[[\epsilon]]$ by defining $\partial_0(\epsilon) := \epsilon$ and $\partial_n(\epsilon) := 0 $ for $n > 0$, we obtain a second exponential homomorphism $E_\eta:R[[\epsilon]] \to R[[\epsilon]][[\eta]]$.  On the other hand, there is a natural continuous homomorphism $\Delta:R[[\zeta]] \to R[[\epsilon]][[\eta]]$ given by $\zeta \mapsto (\epsilon + \eta)$.  Expanding the powers of $(\epsilon + \eta)$, one sees easily that the iteration rule holds if and only if $\Delta \circ E_\zeta = E_\eta \circ E_\epsilon$.  That is, the following diagram is commutative.

$$\begin{CD}  R @>{E_\epsilon}>> R[[\epsilon]] \\ @V{E_\zeta}VV   @VV{E_\eta}V \\ R[[\zeta]] @>{\zeta \mapsto (\epsilon + \eta)}>> R[[\epsilon]][[\eta]] \end{CD}$$

We generalise this ring-theoretic treatment of iterative Hasse-Schmidt differential rings to produce a theory of generalised Hasse-Schmidt rings by encoding the generalised Leibniz rules via exponential maps and the iteration rules via a commutative diagram analogous to the one describing the iteration rule for Hasse-Schmidt derivations.   To present a notion of an iterative Hasse-Schmidt ring we need two kinds of data.  First, we need a projective system of finite free ring schemes $\underline{\cD} := ( \pi_{i,j}:\cD_i \to \cD_j \ | \ 0 \leq j \leq i<\omega)$.  That is, we ask that each $\cD_i$ is, as an additive group scheme, simply some finite cartesian power of the usual additive group scheme while multiplication is given by some regular functions.  A $\underline{\cD}$-ring structure on $R$ is then given by a sequence of ring homomorphisms $E_i:R \to \cD_i(R)$ which are compatible with the projective system.
So in the differential setting $\D_i(R)$ was $R[\epsilon]/(\epsilon^{i+1})$ and $E_i$ was $a\mapsto\sum_{n=0}^i\partial_n(a)\epsilon^n$.
In this general setting, fixing the identifications of each $\cD_i$ with a power of the additive group, the map $E_i$ may be presented as $x \mapsto (\partial_0^{(i)}(x),\ldots,\partial_{m_i}^{(i)}(x))$ where each $\partial_k^{(i)}:R \to R$ is an additive operator.  To say that these operators give $R$ a $\underline{\cD}$-ring structure is equivalent to imposing certain generalised Leibniz rules and identities relating the components of $E_i$ to those of $E_j$.   The second kind of data we require is a collection of morphisms of ring schemes $\Delta_{i,j}:\cD_{i+j} \to \cD_i \circ \cD_j$.  For $(R,\langle E_i:i\in\mathbb N\rangle)$ to be an iterative $\underline\D$-ring we require the following diagrams to commute.
$$\begin{CD}  R @>{E_j}>> \cD_j(R) \\ @V{E_{i+j}}VV   @VV{\cD_j(E_i)}V \\ \cD_{i+j}(R) @>{\Delta_{i,j}}>> \cD_i(\cD_j(R)) \end{CD}$$
We were led to this notion of iteration by considering Matsumura's presentation of
the theory for Hasse-Schmidt derivations.

This theory of generalised iterative Hasse-Schmidt rings is developed in Section~\ref{genhasssys}.
In the appendix we discuss several examples, other than the differential one, showing that this formalism captures many of the interesting cases of rings with distinguished operators.

Our main goal is to understand algebraic equations involving Hasse-Schmidt operators and these equations are naturally encoded by Hasse-Schmidt schemes, or really, Hasse-Schmidt subschemes of algebraic schemes.   To make the issues more concrete, a $\underline{\cD}$-equation in some $\underline{\cD}$-ring $R$ is simply an algebraic equation on the variables and several of the operators $\partial_k^{(i)}$ applied to the variables.  As such, the set of solutions naturally forms a subset of the $R$-points of some algebraic scheme $X$ and the equations themselves are encoded by projective systems of subschemes of prolongation spaces of $X$.  We shall refer to these projective systems as $\underline{\cD}$-schemes.
They are studied in some detail in Section~\ref{sec-hassescheme}.

If $X$ is an algebraic variety over a field $k$, then by the $n$th jet space of $X$ at a point $p\in X(k)$ we mean the space $\hom_k\big(\mathfrak m_{X,p}/\mathfrak m_{X,p}^{n+1},k\big)$.
In Section~\ref{sec-hassejet} we define jet spaces for $\underline{\cD}$-varieties and show that they have enough points to distinguish between different $\underline{\cD}$-subvarieties, at least under an appropriate separability hypothesis.  We have already encountered the main difficulty in~\cite{paperA}; the prolongation space and jet space functors do not commute.
However, in that paper, and based on a prototype already appearing in the work of Pillay and Ziegler (section~5 of~\cite{pillayziegler03}), we introduced an {\em interpolation map} which compares the jet space of a prolongation with the prolongation of a jet space.
This is the key technical ingredient in our construction of jet spaces for $\underline{\cD}$-varieties.

To close this introduction, let us be clear about our aims in the present
paper.  We develop the geometry of algebraic equations involving
additional operators.   While our setting may be regarded as a generalisation
of difference, differential, and Hasse-Schmidt differential algebra, our main goal is
to unify these subjects rather than to generalise them (though our formalism
does allow for such a generalisation).  This unification manifests itself not
only in proofs and constructions which apply equal well to each of the
principal examples, but in a precise formalism for studying confluence
between Hasse-Schmidt differential and difference algebraic geometry.   In terms of
the geometry, our primary goal is to make sense of the linearisation of
general $\underline{\mathcal D}$-equations through a jet space construction and
then to show that these linear spaces determine the $\underline{\mathcal
D}$-varieties, at least under suitable separability hypotheses.
By the Krull intersection theorem this last point is a tautology for
algebraic varieties, but it is far from obvious even when one specialises to 
a well-known theory of fields with operators such as partial difference or differential
algebra.
For {\em finite-dimensional} difference/differential varieties, jet spaces were constructed by Pillay and Ziegler~\cite{pillayziegler03}.
Our theory extends theirs to the infinite-dimensional setting.

In the present paper, we do not develop the model theory of general
$\underline{\mathcal D}$-fields and leave such questions as the existence of model
companions, simplicity, the behaviour of ranks,
\emph{et cetera} to a later work.  Jet spaces were the key technical devices
of the Pillay-Ziegler geometric proofs of the dichotomy theorem for minimal types in differentially closed fields of characteristic zero.
In~\cite{MPS}, arc spaces substituted for jet spaces to extend the dichotomy
theorem to {\em regular} types.
While arc spaces did the job in the differential case, jet spaces are preferable because they give a direct linearisation of the equations.
Provided that
the foundational model-theoretic issues are resolved, our theorem on
$\underline{\mathcal D}$-jet spaces determining $\underline{\mathcal D}$-varieties should give information about canonical bases of (quantifier-free) types in the corresponding theory of $\underline\cD$-fields.

Likewise, there are some closely allied algebraic issues we do not pursue
here.  For example,
jet spaces are clearly connected to a general theory of $\underline{\mathcal
D}$-modules. Moreover, we have not fleshed out the theory of specialisations
of $\underline{\mathcal D}$-rings nor in its local form a theory of valued $\underline{\mathcal
D}$-fields.  Each of these further developments motivates our research into
jet spaces for Hasse-Schmidt varieties and will be taken up in future work.

We are very grateful to the
referee for making some very helpful suggestions.

\section{Generalised Hasse-Schmidt rings}
\label{genhasssys}

\noindent
Let us recall the following conventions and definitions from~\cite{paperA}.   In this paper, all our rings are commutative and unitary and all our ring homorphisms preserve the identity. All schemes are separated.
A {\em variety} is a reduced scheme of finite-type over a field, but is not necessarily irreducible.
Throughout this paper we fix a ring $A$ and work in the categories of $A$-algebras and schemes over $A$.

The {\em standard ring scheme} $\mathbb{S}$ over $A$ is the scheme $\spec\big(A[x]\big)$ endowed with the usual ring scheme structure.
So for all $A$-algebras $R$, $\mathbb{S}(R)=(R,+,\times,0,1)$.
An {\em $\mathbb{S}$-algebra scheme} $\E$ over $A$ is a ring scheme together with a ring scheme morphism $s_\E:\mathbb{S}\to\E$ over $A$.
We view $\mathbb{S}$ as an $\mathbb{S}$-algebra via the identity $\id:\mathbb{S}\to\mathbb{S}$.
A {\em morphism} of $\mathbb{S}$-algebra schemes is then a morphism of ring schemes respecting the $\mathbb{S}$-algebra structure.
Similarly one can define {\em $\mathbb{S}$-module} schemes and morphisms.

\begin{definition}
\label{finitefree}
By a {\em finite free $\mathbb{S}$-algebra scheme with basis} we mean an $\mathbb{S}$-algebra scheme $\mathcal{E}$ together with an $\mathbb{S}$-module isomorphism $\psi_\E:\mathcal{E}\to\mathbb{S}^\ell$, for some $\ell\in\mathbb{N}$.
\end{definition}

The data of a finite free $\mathbb{S}$-algebra scheme with basis is really nothing more than a finite free $A$-algebra with an $A$-basis.
Indeed, fixing $\psi_{\mathcal{E}}$ means that we have a canonical choice of basis $\{1,e_1,\dots,e_{\ell-1}\}$ for $\mathcal{E}(A)$ over $A$.
Write $\displaystyle e_ie_j=\sum_{k=1}^{\ell-1}a_{i,j,k}e_k$ where $a_{i,j,k}\in A$.
So for any $A$-algebra $R$,
$\E(R)$ is the $R$-algebra $R[X_1,\dots,X_{\ell-1}]/I$ where $I$ is generated by polynomials of the form $\displaystyle X_iX_j-\sum_{k=1}^{\ell-1}a_{i,j,k}X_k$.
This means that we can canonically identify $\mathcal{E}(R)$ with $R\otimes_A\mathcal{E}(A)$, both as an $R$-algebra and an $\mathcal{E}(A)$-algebra.
In particular, $\mathcal E$ is determined by $\E(A)$.
Conversely, every finite free $A$-algebra with an $A$-basis naturally determines a finite free $\mathbb S$-algebra scheme with basis.
Indeed, if the given $A$-basis of $B$ is $\{b_0,\dots,b_{\ell-1}\}$ then write $\displaystyle b_ib_j=\sum_{k=0}^{\ell-1}a_{i,j,k}b_k$ and let $\E$ be the $\mathbb S$-algebra scheme whose underying scheme is $\mathbb A_A^\ell$, addition is co-ordinatewise, and comultiplication $A[Z_0,\dots,Z_{\ell-1}]\to A[Z_0,\dots,Z_{\ell-1}]\otimes_AA[Z_0,\dots,Z_{\ell-1}]$ is given by $\displaystyle Z_k\mapsto \sum_{i,j}a_{i,j,k}(Z_i\otimes Z_j)$.
Then $\E(A)=B$.

\smallskip

Given a finite free $\mathbb{S}$-algebra scheme $\E$, an {\em $\E$-ring} is an $A$-algebra $k$ together with an $A$-algebra homomorphism $e:k\to\E(k)$.
A detailed study of $\E$-rings was carried out in~\cite{paperA}, and we will assume the results of that paper in what follows.
We are interested here in rings equipped with an entire directed system of $\E$-ring structures for various $\E$.
The following definition of a Hasse-Schmidt system is a variation on Definition~2.1.1 of the second author's PhD thesis~\cite{Sc-thesis}, differing in a few important details.

\begin{definition}[Hasse-Schmidt system]
\label{hasse-system}
A {\em generalised Hasse-Schmidt system} over $A$ is a projective system of finite free $\mathbb{S}$-algebra schemes with bases over $A$,
$$\underline{\D}=(\pi_{m,n}:\D_m\to\D_n\ | \ n\leq m<\omega),$$
such that $\D_0=\mathbb{S}$
and the transition maps $\pi_{m,n}$ are surjective ring scheme morphisms over $A$.
We denote by $s_n:\mathbb{S}\to\D_n$ the $\mathbb{S}$-algebra structure on $\D_n$ and by $\psi_n:\D_n\to\mathbb{S}^{\ell_n}$ the $\mathbb{S}$-module isomorphisms witnessing a basis for $\D_n$.
\end{definition}

\begin{remark}
\label{noscheme}
The use of scheme-theoretic language for describing Hasse-Schmidt systems is mostly a matter of taste and convenience; it can easily be avoided.
Indeed, by the discussion following Definition~\ref{finitefree} above, evaluating at $A$ yields a bijective correspondence between Hasse-Schmidt systems and projective systems of finite free $A$-algebras equipped with $A$-bases.
\end{remark}

\begin{definition}[Hasse-Schmidt ring]
\label{dring}
Suppose $\underline{\D}$ is a Hasse-Schmidt system over $A$.
A {\em generalised Hasse-Schmidt ring} (or {\em $\underline{\D}$-ring}) {\em over $A$} is an $A$-algebra equipped with a system of $\D_n$-ring structures that are compatible with $\pi$.
That is, a $\underline{\D}$-ring is a pair $(k,E)$ where $k$ is an $A$-algebra and $E=(E_n:k\to\D_n(k) \ | \ n\in\mathbb{N})$ is a sequence of $A$-algebra homomorphisms such that
\begin{itemize}
\item[(i)]
$E_0=\id$,
\item[(ii)]
the following diagram commutes for all $m\geq n$
$$\xymatrix{
\D_m(k)\ar[rr]^{\pi_{m,n}^k} && \D_n(k)\\
& k\ar[ul]^{E_m}\ar[ur]_{E_n}
}$$
\end{itemize}
\end{definition}

\begin{remark}
\begin{itemize}
\item[(a)]
One may equally well describe a $\underline \D$-ring by giving a collection of maps
$(\partial_{i,n}:k \to k \ |\ n \in {\mathbb N}, i \leq \ell_n)$ via the correspondence
$\psi_n \circ E_n = (\partial_{1,n},\ldots,\partial_{\ell_n,n})$.  That the
collection $(\partial_{i,n})$ so defines a $\underline \D$-ring structure on $k$ is equivalent to
the satisfaction of a certain system of functional equations.
\item[(b)]
Our choice of a natural-number-indexing for Hasse-Schmidt systems is convenient but not absolutely necessary.
Indeed, some contexts may be more naturally dealt with by considering Hasse-Schmidt systems indexed by $\mathbb{N}^r$ or even $\mathbb{Z}^r$.
However, indexing by $\mathbb{N}$ does simplify the exposition somewhat, and all our examples can be made to fit into this setting.
\end{itemize}
\end{remark}

The first example of a Hasse-Schmidt system is where each $\D_n=\mathbb{S}$ and $\pi_{m,n}=\psi_n=\id$.
Then for any $A$-algebra $k$, the only $\underline{\D}$-ring structure on $k$ is the trivial one with $E_n=\id$.
This example captures the context of rings without any additional structure.
Our main example, that of Hasse-Schmidt differential rings, is discussed below.
See the Appendix for a discussion of several other examples including difference rings, an analogue of $q$-iterative difference rings, and difference-differential rings.

\begin{example}[Hasse-Schmidt differential rings]
\label{differential}
Here $A=\mathbb Z$.
Consider the Hasse-Schmidt system $\operatorname{HSD}_e=(\pi_{m,n}:\D_m\to\D_n\ | \ n\leq m<\omega)$ where for any ring $R$
\begin{itemize}
\item
$\D_n(R)=R[\eta_1,\dots,\eta_e]/(\eta_1,\dots,\eta_e)^{n+1}$, where $\eta_1,\dots,\eta_e$ are indeterminates;
\item
$s_n^R:R\to\D_n(R)$ is the natural inclusion;
\item
$\psi_n^R: \D_n(R)\to R^{\ell_n}$ is an identification via a fixed ordering of the monomial basis of $R[\eta_1,\dots,\eta_e]/(\eta_1,\dots,\eta_e)^{n+1}$ over $R$; and,
\item
for $m\geq n$, $\pi_{m,n}^R:\D_m(R)\to\D_n(R)$ is the quotient map.
\end{itemize}
Note that this does uniquely determine a Hasse-Schmidt system (even considering only $R=\mathbb Z$, see Remark~\ref{noscheme}).
Writing $\displaystyle E_n(x)=\sum_{\alpha\in\mathbb{N}^e,|\alpha|\leq n}\partial_{\alpha}(x)\eta^{\alpha}$, an $\operatorname{HSD}_e$-ring is a ring $k$ together with a sequence of additive maps $(\partial_{\alpha}:k \to k \ |\ \alpha\in\mathbb N^e)$ satisfying
$\displaystyle \partial_{\alpha}(xy)=\sum_{\beta+\gamma=\alpha}\partial_{\beta}(x)\partial_{\gamma}(y)$,
$\partial_{\overline{0}}=\id$, and
$\partial_{\alpha}(1)=0$ for $|\alpha|>0$.

The principal example of an $\operatorname{HSD}_e$-ring is a ring equipped with $e$ Hasse-Schmidt derivations.
Recall that a {\em Hasse-Schmidt derivation} on a ring $k$ is a sequence of additive maps from $k$ to $k$, ${\bf D}=(D_0,D_1\dots)$, such that
\begin{itemize}
\item
$D_0=\id$ and
\item
$\displaystyle D_n(xy)=\sum_{a+b=n}D_a(x)D_b(y)$.
\end{itemize}
(cf. Section~27 of~\cite{matsumura}, for example.)
Suppose ${\bf D}_1,\dots,{\bf D}_e$ is a sequence of $e$ Hasse-Schmidt derivations on $k$ and set $\displaystyle E(x)=\sum_{\alpha\in\mathbb{N}^e} D_{1,\alpha_1}D_{2,\alpha_2}\cdots D_{e,\alpha_e}(x)\eta^{\alpha}$.
Then
$E:k\to k[[\eta_1,\dots,\eta_e]]$
is a ring homomorphism and we can view it as a system $(E_n \ |\ n<\omega)$ where $E_n$ is the composition of $E$ with the quotient $k[[\eta_1,\dots,\eta_e]]\to k[\eta_1,\dots,\eta_e]/(\eta_1,\dots,\eta_e)^{n+1}$.
Then $(k,E)$ is an $\operatorname{HSD}_e$-ring.

This example specialises further to the case of {\em partial differential fields in characteristic zero}.
Suppose $k$ a field of characteristic zero and $\partial_1,\dots,\partial_e$ are derivations on $k$.
Then $\displaystyle D_{i,n}:=\frac{\partial_i^n}{n !}$, for $1\leq i\leq e$ and $n\geq 0$, defines a sequence of Hasse-Schmidt derivations on $k$.
The $\operatorname{HSD}_e$-ring structure on $k$ is given in multi-index notation by $\displaystyle E_n(x):=\sum_{\alpha\in\mathbb{N}^e, |\alpha|\leq n}\frac{1}{\alpha!}\partial^\alpha(x)\eta^\alpha$ where $\partial:=(\partial_1,\dots,\partial_e)$.

On the other hand we can specialise in a different direction to deal with {\em fields of finite imperfection degree}.
The following example is informed by~\cite{ziegler03}:
suppose $k$ is a field of characteristic $p>0$ with imperfection degree $e$.
Let $t_1,\dots, t_e$ be a $p$-basis for $k$.
Consider $\mathbb{F}_p[t_1,\dots,t_e]$ and for $1\leq i\leq e$ and $n\in\mathbb{N}$, define
$${\bf D}_{i,n}(t_1^{\alpha_1}\cdots t_e^{\alpha_e}):=\left(\begin{array}{c}\alpha_i\\n\end{array}\right)t_1^{\alpha_1}\cdots t_i^{\alpha_i-n}\cdots t_e^{\alpha_e}.$$
and extend by linearity to $\mathbb{F}_p[t_1,\dots,t_e]$.
Then $({\bf D}_1,\dots,{\bf D}_e)$ is a sequence of Hasse-Schmidt derivations on $\mathbb{F}_p[t_1,\dots,t_e]$.
Moreover, they extend uniquely to Hasse-Schmidt derivations on $k$ (see Lemma~2.3 of~\cite{ziegler03}).
This gives rise to an $\operatorname{HSD}_e$-ring structure on $k$.

It is not the case that every $\operatorname{HSD}_e$-ring is a Hasse-Schmidt differential ring.
In section~\ref{subsect-iterativity} below we will introduce the notion of {\em iterativity} for Hasse-Schmidt systems, which when applied to this example will allow us to capture exactly the (commuting and iterative) Hasse-Schmidt differential rings.
\end{example}

\subsection{Hasse-Schmidt prolongations}
A generalised Hasse-Schmidt structure on a ring $k$ induces, for every algebraic scheme $X$ over $k$, a sequence of (abstract) prolongations of $X$ in the sense of~\cite{paperA}.
We recall the construction here.

Fix a generalised Hasse-Schmidt system $\underline\D$ over $A$, and a $\underline\D$-ring $(k,E)$.

\begin{definition}[The exponential algebra structure]
For each $n$, by the {\em exponential} $k$-algebra structure on the underlying ring of $\D_n(k)$ we mean the $k$-algebra structure coming from the ring homomorphism $E_n:k\to\D_n(k)$.
We denote this $k$-algebra by $\D_n^{E_n}(k)$.
More generally, given any $k$-algebra $a:k\to R$,
$\D_n^{E_n}(R)$ denotes the {\em exponential} $k$-algebra given by
$$\xymatrix{
k\ar[rr]^{E_n} && \D_n(k)\ar[rr]^{\D_n(a)} &&\D_n(R)
}$$
Note that as $A$-algebras, $\D_n^{E_n}(R)$ and $\D_n(R)$ are identical.
\end{definition}

\begin{definition}[Prolongations]
\label{defprolong}
Suppose $X$ is a scheme over $k$.
The {\em $n$th prolongation of $X$}, $\tau(X,\D_n,E_n)$, or just $\tau_n X$ for short, is the Weil restriction of $X\times_k\D_n^{E_n}(k)$ from $\D_n(k)$ to $k$ (when it exists).
We usually write $\tau X$ for $\tau_1 X$.
Note that the base extension is with respect to the exponential $k$-algebra, while the Weil restriction is with repect to the standard $k$-algebra.
\end{definition}

See $\S 2$ of~\cite{paperA} for details on the Weil restriction functor.
In particular, it follows from the discussion there (and it is well known) that the $\D_n(k)$ over $k$ Weil restriction functor takes affine schemes over $\D_n(k)$ to affine schemes over $k$.
As the base change functor also preserves affine schemes, we see that the prolongation of an affine scheme over $k$ is again an affine scheme over $k$.

But what is the prolongation really?
By Lemma~4.5 of~\cite{paperA}, its characteristic property is that for any $k$-algebra $R$, there is a canonical identification
\begin{equation}
\label{prolongpoint}
\tau_n X(R)=X\big(\D_n^{E_n}(R)\big).
\end{equation}
Indeed, recall that coming from the definition of prolongations via Weil restrictions we have an $n$th {\em canonical morphism} $r^X_n:\tau_n X\times_k\D_n(k)\to X$ for each $n\in\mathbb{N}$.
The identification $\tau_n X(R)=X\big(\D_n^{E_n}(R)\big)$ is then
by $p\mapsto r^X_n\circ\big(p\times_k\D_n(k)\big)$.
See Definition~4.3 of~\cite{paperA} for details.
One thing to remark is that $r^X_n$ is not over $k$ in the usual manner, rather we have the commuting diagram
\label{canonicalmap}
$$\xymatrix{
\tau_n X\times_k\D_n(k)\ar[rr]^{r^X_n}\ar[dr] && X\ar[dd]\\
& \spec\big(\D_n(k)\big)\ar[dr]_{\spec(E_n)} &\\
&& \spec(k)
}$$
where $\spec(E_n)$ is the morphism of schemes induced by $E_n:k\to\D_n(k)$.
Put another way, working at the level of co-ordinate rings, $r^X_n$ induces a $k$-algebra morphism from $k[X]$ to $\D_n^{E_n}\big(k[\tau_n X]\big)$.

\begin{remark}
\label{adjoint}
The prolongation functor thus induces a functor on $k$-algebras, assigning to the co-ordinate ring of an affine $k$-scheme $X$ the co-ordinate ring of the affine $k$-scheme $\tau_n X$, which is left-adjoint to the functor $R\mapsto\D_n^{E_n}(R)$.
Indeed, this is exactly what the displayed identity~(\ref{prolongpoint}) above asserts.
\end{remark}

\begin{definition}[The nabla map]
\label{nabla-def}
For $X$ a scheme over $k$, under the above identification, $E_n:k\to\D_n^{E_n}(k)$ induces a map $\nabla_n:X(k)\to\tau_n X(k)$.
\end{definition}

\begin{remark}
\begin{itemize}
\item[(a)]
It is not always the case that the Weil restriction, and hence the prolongation, exists.
However, if $X$ is such that every finite set of points in $X$ is contained in an affine open subscheme, then $\tau_n X$ does exist.
In particular prolongations of quasi-projective schemes always exist.
For more details on Weil restrictions see Section~2 of~\cite{paperA}.
\item[(b)]
Definition~\ref{defprolong} is just the definition of an abstract prolongation (Definition~4.1 of~\cite{paperA}), specialised to the finite free $\mathbb{S}$-algebra schemes $\D_n$.
It follows from the work in that paper that $\tau_n$ is a covariant functor which preserves \'etale morphisms, smooth embeddings, and closed embeddings (cf. Proposition~4.6 of~\cite{paperA}).
\item[(c)]
The nabla map is only defined on $k$-points, or more generally on points lying in $\underline\D$-ring extension of $(k,E)$, and not in arbitrary $k$-algebras.
It is a ``$\underline\D$-algebraic'' map.
\end{itemize}
\end{remark}

\begin{example}
\label{prolongdifferential}
In our main examples the prolongation spaces specialise to the expected objects.
So for pure rings (when $\D_n=\mathbb{S}$) we get $\tau_n X=X$.
For rings equipped with endomorphisms $\sigma_i$ (this is Example~\ref{difference} of the Appendix, when $\underline\D=\operatorname{End}$),
$\tau_nX=X\times X^{\sigma_1}\times\cdots\times X^{\sigma_n}$, and $\nabla_n(x)=\big(x,\sigma_1(x),\dots,\sigma_n(x)\big)$.
In the Hasse-Schmidt differential case of Example~\ref{differential}, the $\tau_n X$ and $\nabla_n$ are the usual differential prolongations with their differential sections.
For example, if $(k,\delta)$ is an ordinary differential field of characteristic zero, then $\nabla_n(x)=\big(x,\delta(x),\dots,\frac{\delta^n(x)}{n!}\big)$.
See Example~4.2 of~\cite{paperA} for more details on these particular cases.
It is also worth pointing out that if we take $\delta=0$ then $\tau_n X$ is the $n$th {\em arc space} of $X$, $\operatorname{Arc}_n X$, and $\nabla_n$ is the zero section.
See~\cite{denefloeser01} for a survey on arc spaces.
(They are the higher tangent bundles; $\operatorname{Arc}_1 X$ is the tangent bundle of $X$.)
\end{example}

For $m\geq n$, the morphisms $\pi_{m,n}:\D_m\to\D_n$ induce morphisms $\hat\pi_{m,n}:\tau_m X\to\tau_n X$.
Indeed, since $k$ is a $\underline\D$-ring, we have that $\pi^k_{m,n}:\D_{m}^{E_{m}}(k)\to\D_n^{E_n}(k)$ is a $k$-algebra homomorphism, and so, for any fixed $k$-algebra $R$, so is the corresponding $\pi^R_{m,n}:\D_{m}^{E_{m}}(R)\to\D_n^{E_n}(R)$.
Now on $R$-points, using the identification~(\ref{prolongpoint}) above, $\hat\pi_{m,n}$ is just the map induced by $\pi^R_{m,n}$.
See section~4.1 of~\cite{paperA} for more details on the morphism between prolongations induced by a morphism of finite free $\mathbb{S}$-algebra schemes.
Setting $m=0$ we see that the $n$th prolongation obtains the structure of a scheme over $X$; namely, $\hat\pi_{n,0}:\tau_n X\to X$.

\begin{lemma}
\label{nablaprop}
Suppose $X$ is a scheme over $k$.
For each $n<\omega$, $\nabla_n$ is a section to $\hat\pi_{n,0}^k:\tau_n X(k)\to X(k)$ and satisfies $\hat\pi_{n+1,n}\circ \nabla_{n+1}=\nabla_n$.
\end{lemma}

\begin{proof}
Immediate from the definitions.
\end{proof}

\begin{proposition}
\label{liftinglemma}
Suppose $k$ is a field and $X$ is a variety (so reduced and of finite-type).
For all $m\geq n$, $\hat\pi_{m,n}:\tau_{m} X\to\tau_n X$ is a dominant morphism.
\end{proposition}

\begin{proof}
This is well known in our main examples, including arc spaces (cf.~\cite{denefloeser01}), difference- and differential prolongations.
The  usual proofs in those cases extend to this setting, but we nevertheless give some details.

Let $K=k^{\alg}$ be the algebraic closure of $k$.
On $K$-points $\hat\pi_{m,n}$ is the map $X\big(\D_{m}^{E_{m}}(K)\big)\to X\big(\D_n^{E_n}(K)\big)$ induced by $\pi_{m,n}:\D_{m}^{E_{m}}(K)\to \D_n^{E_n}(K)$.
Hence the proposition will follow from the following general claim:

\begin{claim}
\label{liftart}
If $\rho:R \to S$ is a surjective map of artinian $K$-algebras and
$P \in X(S)$ is a smooth $S$-point of $X$, then there is an $R$-point $Q \in X(R)$ sent to $P$ by the map induced by $\rho$.
\end{claim}

\begin{proof}[Proof of Claim~\ref{liftart}]
First of all, we can decompose $R$ and $S$ as products of artinian $K$-algebras, $R \cong (\prod_{i=1}^n A_i) \times C$ and $S \cong \prod_{i=1}^n B_i$, where the $A_i$s and $B_i$s are local, and there exist local surjective homomorphisms $\rho_i:A_i\to B_i$, such that for all $x=(a_1,\dots,a_n,c)\in R$, $\rho(x)=(\rho_1(a_1),\dots,\rho_n(a_n)\big)$.
Now for each $i\leq n$, let $P_i\in X(B_i)$ be the image of $P$ under the map $X(S)\to X(B_i)$ induced by the projection $S\to B_i$.
Since $A_i$ is artinian and $\rho_i:A_i\to B_i$ is local, $\ker(\rho_i)$ is nilpotent.
Hence we can lift the smooth $B_i$-point $P_i$ of $X$, to a point $Q_i\in X(A_i)$.
Indeed, if $\ker(\rho_i)$ were square zero this would be the definition of smoothness; by induction it holds for nilpotent kernel also.
As $K$ is algebraically closed, we
can find $Q_C \in X(C)$.
Now, letting $Q \in X(R)$ be the point which projects to $Q_C \in X(C)$ and $Q_i \in X(A_i)$ for $i\leq n$, we get that $\rho$ maps $Q$ to $P$ as desired.
\end{proof}

We complete the proof of Proposition~\ref{liftinglemma}.
Using the functoriality of the prolongations, we may assume that $X$ is irreducible over $k$.
Now, by the claim, every smooth point of $X\big(\D_n^{E_n}(K)\big)$ is in the image of $X\big(\D_{m}^{E_{m}}(K)\big)\to X\big(\D_n^{E_n}(K)\big)$.
Let $Y$ be the proper $k$-closed subvariety of singular points  of $X$.
Then under the identification $X\big(\D_n^{E_n}(K)\big)=\tau_n X(K)$, the set $Y\big(\D_n^{E_n}(K)\big)$ is identified with $\tau_n Y(K)$, which is a proper $k$-closed subset of $\tau_n X(K)$.
Hence $\hat\pi_{m,n}:\tau_{m} X\to\tau_n X$ is dominant, as desired.
\end{proof}

We record the following fact from ~\cite{paperA} for later use:

\begin{fact}[Proposition~4.7(b) of~\cite{paperA}]
\label{nablafunctorialatpoints}
Suppose $f:X\to Y$ is a morphism of schemes over $k$ and $a\in Y(k)$.
Then $(\tau_n X)_{\nabla_n(a)}$, the fibre of $\tau_n(f):\tau_n X\to\tau_n Y$ over $\nabla_n(a)$, is $\tau_n(X_a)$.
\qed
\end{fact}

\subsection{Iterativity}
\label{subsect-iterativity}
As explained in Section~4.2 of~\cite{paperA} we can compose finite free $\mathbb{S}$-algebra schemes.
Specialising to Hasse-Schmidt systems, for all $m,n\in\mathbb{N}$ we get finite free $\mathbb{S}$-algebra schemes $\D_{(m,n)}:=\D_m\D_n$.
So for any $A$-algebra $R$, $\D_{(m,n)}(R)=\D_m\big(\D_n(R)\big)$ where the $R$-algebra structure is given by
$$\xymatrix{
R\ar[rr]^{s^R_n}&&\D_n(R)\ar[rr]^{s^{\D_n(R)}_m}&&\D_m\big(\D_n(R)\big).
}$$
By Remark~4.10 of~\cite{paperA} we know that $\D_{(m,n)}=\D_m\D_n$ is canonically isomorphic to $\D_m\otimes_{\mathbb S}\D_n$.
There are also the $A$-algebra homomorphisms $E_{(m,n)}:k\to\D_{(m,n)}(k)$, given by
$$\xymatrix{
k\ar[rr]^{E_m}&&\D_m(k)\ar[rr]^{\D_m(E_n)}&&\D_m\big(\D_n(k)\big).
}$$
What Proposition~4.12 of~\cite{paperA} tells us is that
$\tau_n(\tau_m X)=\tau(X,\D_{(m,n)},E_{(m,n)})$ and
$\nabla_n\circ\nabla_m=\nabla_{\D_{(m,n)},E_{(m,n)}}$.
Note that in this context,
for $m'\leq m$ and $n'\leq n$, we have the ring scheme morphisms $\pi_{(m,n),(m',n')}:\D_{(m,n)}\to\D_{(m',n')}$ given by the composition
$$\xymatrix{\D_m\big(\D_n(R)\big)\ar[rr]^{\D_m(\pi_{n,n'}^R)} && \D_m\big(\D_{n'}(R)\big)\ar[rr]^{\pi_{m,m'}^{\D_{n'}(R)}} && \D_m'\big(\D_{n'}(R)\big).}$$

It is a matter of fact that all the examples of Hasse-Schmidt rings corresponding to the various Hasse-Schmidt systems that we are particularly interested in satisfy some further relations not implied by the definition of being a Hasse-Schmidt ring.
These further relations can be viewed as certain iterativity conditions relating $E_{(m,n)}$ with $E_{m+n}$.
We formalise this as follows.

\begin{definition}
\label{iterative-hasse-system}
An {\em iterative} Hasse-Schmidt system is a Hasse-Schmidt system $\underline\D$ together with a sequence of closed embeddings of ring schemes
$$\Delta=\big(\Delta_{(m,n)}:\D_{m+n}\to\D_{(m,n)}\big)_{m,n\in\mathbb{N}}$$
such that:
\begin{itemize}
\item[(a)]
$\Delta$ is compatible with $\pi$.
That is, for all $m'\leq m$ and $n'\leq n$, the following diagram commutes:
$$\xymatrix{
\D_{m+n}\ar[d]_{\pi_{m+n,m'+n'}}\ar[rr]^{\Delta_{(m,n)}} && \D_{(m,n)}\ar[d]^{\pi_{(m,n),(m',n')}}\\
\D_{m'+n'}\ar[rr]^{\Delta_{(m',n')}} &&\D_{(m',n')}
}$$
\item[(b)]
$\Delta$ is associative in the sense that
for all $\ell,m,n$, and any $A$-algebra $R$,
$$\xymatrix{
\D_\ell\big(\D_{m+n}(R)\big)\ar[rr]^{\D_\ell(\Delta^R_{(m,n)})}&&\D_\ell\big(\D_m(\D_n(R))\big)\\
\D_{\ell+m+n}(R)\ar[u]^{\Delta^R_{(\ell,m+n)}}\ar[rr]^{\Delta^R_{(\ell+m,n)}}&&\D_{\ell+m}\big(\D_n(R)\big)\ar[u]_{\Delta_{(\ell,m)}^{\D_n(R)}}
}$$
commutes.
\item[(c)]
$\Delta_{(m,0)}=\Delta_{(0,n)}=\id$ for all $m,n\geq 0$.
\end{itemize}
We say that $(k,E)$ is an {\em iterative Hasse-Schmidt ring} (or more accurately $\Delta$-iterative) if it is a $\underline\D$-ring and
$$\xymatrix{
\D_{m+n}(k)\ar[rr]^{\Delta_{(m,n)}^k} && \D_m\left(\D_n(k)\right)\\
& k\ar[ul]^{E_{m+n}}\ar[ur]_{E_{(m,n)}}
}$$
commutes for all $m, n\in\mathbb{N}$.
That is, $\Delta_{(m,n)}^k:\D_{m+n}^{E_{m+n}}(k)\to\D_{(m,n)}^{E_{(m,n)}}(k)$ is a $k$-algebra map for all $m, n\in\mathbb{N}$.
\end{definition}

Our definition of an iterative
Hasse-Schmidt system was inspired by the presentation of the iteration rules
for higher derivations in~\cite{matsumura}.   Other authors have considered
similar (and in some cases even more general) notions.  If one were to
replace our system of ring functions by their projective limit
$\mathcal D_\infty := \varprojlim {\mathcal D}_n$ and the iteration
system by a single natural transformation $\Delta:\mathcal D_\infty \to
\mathcal D_\infty \circ \mathcal D_\infty$, then our axioms may be
read as saying that $\mathcal D_\infty$ is a comonad on the category of
$k$-algebras and an iterative Hasse-Schmidt ring would be an Eilenberg-Moore
$\mathcal D_\infty$-coalgebra.    This point of view is taken, for
example, in the work of Borger and Wieland on plethories~\cite{borgerweiland}
and of Keigher~\cite{keigher} in the
study of algebraic ${\mathcal D}$-modules.  While there are some conceptual
simplifications to be gained by passing to the inverse limit and using the
theory of (co)monads, we have consciously avoided this move, partly for the sake on concreteness, and partly because for us it is
very important that the prolongation spaces associated to finite type
schemes be themselves of finite type.

\begin{remark}
\label{deltainduced}
For all schemes $X$ over a $\underline\D$-ring $(k,E)$,
the iteration maps induce morphisms
$\hat\Delta_{(m,n)}:\tau_{m+n} X\to\tau_n\tau_m X$
such that the following diagram commutes:
$$
\xymatrix{
\tau_{m+n} X(k)\ar[rr]^{\hat\Delta_{(m,n)}}&&\tau_n\tau_m X(k)\\
& X(k)\ar[ul]^{\nabla_{m+n}}\ar[ur]_{\nabla_n\circ\nabla_m}
}
$$
(cf. Propositions~4.8(a) and~4.12 of~\cite{paperA}).
Moreover, since the iteration maps are closed embeddings, these induced morphisms are also closed embeddings (cf. Proposition~4.8(c) of~\cite{paperA}).
\end{remark}

We will need the following lemma later:

\begin{lemma}
\label{deltamore}
Suppose $(\underline\D,\Delta)$ is an iterative Hasse-Schmidt system.
Then for all $m,n\in\mathbb{N}$, and all $A$-algebras $R$, the following diagram commutes:
$$
\xymatrix{
\D_m\big(\D_{n+1}(R)\big)\ar[rr]^{\D_m(\pi_{n+1,n}^R)} && \D_m\big(\D_n(R)\big)\\
\D_{m+n+1}(R)\ar[u]^{\Delta^R_{(m,n+1)}}\ar[rr]^{\Delta^R_{(m+1,n)}} && \D_{m+1}\big(\D_n(R)\big)\ar[u]_{\pi_m^{\D_n(R)}}
}$$
\end{lemma}

\begin{proof}
This is a combination of the associativity of $\Delta$ together with its compatibility with $\pi$.
We will prove that the desired diagram commutes by proving that three other diagrams commute.
First of all,
\begin{eqnarray}
\label{deltamore-assoc}
\xymatrix{
\D_m\big(\D_{n+1}(R)\big)\ar[rr]^{\D_m(\Delta^R_{(1,n)})} && \D_m\big(\D_1\big(\D_n(R)\big)\big)\\
\D_{m+n+1}(R)\ar[u]^{\Delta^R_{(m,n+1)}}\ar[rr]^{\Delta^R_{(m+1,n)}}
 && \D_{m+1}\big(\D_n(R)\big)\ar[u]_{\Delta_{(m,1)}^{\D_n(R)}}
}
\end{eqnarray}
commutes as it is an instance of Definition~\ref{iterative-hasse-system}(b) (associativity).
Next, note that the following diagram is an instance of Definition~\ref{iterative-hasse-system}(a) with $(1,n)$ and $(0,n)$, using also~\ref{iterative-hasse-system}(c), and hence commutes:
\begin{eqnarray*}
\xymatrix{
\D_{n+1}(R)\ar[rr]^{\Delta^R_{(1,n)}}\ar[d]_{\pi_{n+1,n}^R} && \D_1\big(\D_n(R)\big)\ar[dll]^{\pi_0^{\D_n(R)}}\\
\D_n(R)
}
\end{eqnarray*}
Applying the functor $\D_m$ we get that
\begin{eqnarray}
\label{deltamore-pi1}
\xymatrix{
\D_m\big(\D_{n+1}(R)\big)\ar[rr]^{\D_m(\Delta^R_{(1,n)})}\ar[d]_{\D_m(\pi_{n+1,n}^R)} && 
\D_m\big(\D_1\big(\D_n(R)\big)\big)\ar[dll]^{\D_m\big(\pi_0^{\D_n(R)}\big)}\\
\D_m\big(\D_n(R)\big)
}
\end{eqnarray}
commutes.
Finally, the following is also an instance of Definition~\ref{iterative-hasse-system}(a) with $(m,1)$ and $(m,0)$ applied to the ring $\D_n(R)$, using also~\ref{iterative-hasse-system}(c)
\begin{eqnarray}
\label{deltamore-pi2}
\xymatrix{
&& \D_m\big(\D_1\big(\D_n(R)\big)\big)\ar[dll]_{\D_m\big(\pi_0^{\D_n(R)}\big)}\\
\D_m\big(\D_n(R)\big)
&& \D_{m+1}\big(\D_n(R)\big)\ar[u]_{\Delta_{(m,1)}^{\D_n(R)}}\ar[ll]^{\pi_m^{\D_n(R)}}
}
\end{eqnarray}
Putting the commuting diagrams~(\ref{deltamore-assoc}), (\ref{deltamore-pi1}), and~(\ref{deltamore-pi2}) together proves the lemma.
\end{proof}

We now point out that the Hasse-Schmidt system coming from our main example admits a natural iteration such that the corresponding iterative Hasse-Schmidt rings form exactly the intended class: rings equipped with commuting iterative Hasse-Schmidt derivations.
See the appendix for a discussion of iterativity for other examples.

Consider the Hasse-Schmidt system $\operatorname{HSD}_e$ from Example~\ref{differential}.
So, for $R$ any ring, $\D_n(R)=R[\eta_1,\dots,\eta_e]/(\eta_1,,\dots,\eta_e)^{n+1}$.
We define $\Delta$ so that for all $R$,
$\Delta_{(m,n)}^R$ from
$$\D_{m+n}(R)=R[\eta_1,\dots,\eta_e]/(\eta_1,,\dots,\eta_e)^{m+n+1}$$
to
$$\D_m\big(\D_n(R)\big)=R[\zeta_1,\dots,\zeta_e,\epsilon_1,\dots,\epsilon_e]/(\zeta_1,\dots,\zeta_e)^{n+1}(\epsilon_1,\dots,\epsilon_e)^{m+1}$$
is given by $\eta_i\mapsto(\zeta_i+\epsilon_i)$.

\begin{proposition}
\label{iterativedifferential}
The system $\Delta=(\Delta_{(m,n)}:m,n\in\mathbb{N})$, above, makes $\operatorname{HSD}_e$ into an iterative Hasse-Schmidt system.
The $\Delta$-iterative $\operatorname{HSD}_e$-rings in this case are exactly the rings equipped with $e$ commuting Hasse-Schmidt derivations satisfying the additional identities
$$D_aD_b=\left(\begin{array}{c}a+b\\b\end{array}\right)D_{a+b}$$
for all $a,b,\in\mathbb{N}$.
(Hasse-Schmidt derivations satisfying these identities are called {\em iterative} Hasse-Schmidt derivations.)
\end{proposition}

\begin{proof}
We first observe that $(\operatorname{HSD}_e,\Delta)$ is an iterative system.
Indeed, $\Delta_{m,n}$ is a closed embedding of ring schemes, it is compatible with $\pi$, and it is associative (the latter is just the associativity of $+$), and $\Delta_{(m,0)}=\Delta_{(0,n)}=\id$.

Now suppose $(k,E)$ is an $\operatorname{HSD}_e$-ring.
For each $n$, write $\displaystyle E_n(x)=\sum_{\alpha\in\mathbb{N}^e,|\alpha|\leq n}\partial_{\alpha}(x)\eta^{\alpha}$.
Let $D_{i,n}:=\partial_{(0,\dots,n,\dots,0)}$, where here the multi-index has $n$ in the $i$th co-ordinate and $0$ everywhere else.
So ${\bf D}_1:=(D_{1,0}, D_{1,1},\dots),\dots,{\bf D}_e:=(D_{e,0},D_{e,1},\dots)$ form a sequence of $e$ Hasse-Schmidt derivations.
Now, writing out $\Delta_{(m,n)}\circ E_{m+n}$ using the binomial coefficients, we see that $\Delta$-iterativity in this case is equivalent to
\begin{eqnarray}
\label{iterpart}
\partial_{\alpha}\partial_{\beta}&=&\left(\begin{array}{c}\alpha+\beta\\ \beta\end{array}\right)\partial_{\alpha+\beta}
\end{eqnarray}
for all multi-indices $\alpha$ and $\beta$.
In particular it implies that each ${\bf D}_i$ is an iterative Hasse-Schmidt derivation and that they all commute (indeed all the $\partial_\alpha$ commute).
Conversely, suppose  $\partial_\alpha=D_{1,\alpha_1}\cdots D_{e,\alpha_e}$ for each $\alpha$, and ${\bf D}_1,\dots,{\bf D}_e$ form a sequence of $e$ iterative commuting Hasse-Schmidt derivations.
Then it is not hard to see that ~(\ref{iterpart}) holds and so $(k,E)$ is $\Delta$-iterative.
\end{proof}

\subsection{Jets and interpolation for Hasse-Schmidt prolongations}
\label{subsect-interp}
For a scheme $X$ over a ring $k$, by the {\em $n$th jet space of $X$}, denoted by $\jet^nX\to X$, we mean the linear space\footnote{The linear space associated to a coherent sheaf of $\mathcal O_X$-modules $\mathcal F$ is $\spec(\sym^*\mathcal F)$.
When $\mathcal F$ is locally free the linear space associated to $\mathcal F$ is {\em dual} to the vector bundle associated to $\mathcal F$.} associated to the (coherent) sheaf of $\O_X$-modules $\mathcal{I}/\mathcal{I}^{n+1}$, where $\mathcal{I}$ is the kernel of the map $\O_X\otimes_k\O_X\to\O_X$ given on sections by $f\otimes g\mapsto fg$.
This is a covariant functor, its action on morphisms $f:X\to Y$ being the natural one induced by $f^\sharp:f^{-1}\O_Y\to\O_X$.
More concretely, if $k$ is a field and $p\in X(k)$ then $\jet^nX_p(k)=\hom_k\big(\mathfrak{m}_{X,p}/\mathfrak{m}^{n+1}_{X,p},k\big)$, and $\jet^n(f)_p:\jet^nX_p\to\jet^nY_{f(p)}$ is given by precomposing with $f^\sharp_p:\mathfrak{m}_{Y,f(p)}/\mathfrak{m}^{n+1}_{Y,f(p)}\to \mathfrak{m}_{X,p}/\mathfrak{m}^{n+1}_{X,p}$.
For details we refer the reader to section~5 of~\cite{paperA}, which is dedicated to a review of the relevant properties of this functor.

Jet spaces serve to linearise algebraic varieties in the sense that they can be used to distinguish subvarieties of a given variety: 
Suppose $Z$ and $Z'$ are irreducible subvarieties of an algebraic variety $X$ over a field $k$, and $p\in Z(k)\cap Z'(k)$ with $\jet^n(Z)_p=\jet^n(Z')_p$ for all $n$, then $Z=Z'$.

One of the main purposes of our work in~\cite{paperA} was the introduction of a certain map between the jet space of an abstract prolongation and the prolongation of the jet space.
In fact, a prototype for this map already appears in the work of Pillay and Ziegler (cf. Section~5 of~\cite{pillayziegler03}). 
We now recall the interpolating map specialised to our current setting.
Fix an iterative Hasse-Schmidt system $\underline \D$, an iterative $\underline\D$-ring $(k,E)$, $m\geq 1$, and $n\geq 0$.
Coming from the Weil restriction of scalars functor there is a canonical morphism,
$r:\tau_n X\times_k\D_n(k)\to X\times_k\D_n^{E_n}(k)$, of schemes over $\D_n(k)$.
Applying the jet functor (which commutes with base change) we get
$$\jet^m(r):(\jet^m\tau_n X)\times_k\D_n(k)\to (\jet^mX)\times_k\D_n^{E_n}(k).$$
Applying the restriction of scalars from $\D_n(k)$ to $k$,
we get
$$\res_{\D_n(k)/k}\big(\jet^m(r)\big):\res_{\D_n(k)/k}\big(\jet^m\tau_n X\times_k\D_n(k)\big)\to\tau_n\jet^mX$$
a morphism of schemes over $k$.
Now for any scheme $Y$ over $k$ and any $k$-algebra $R$ there is a 
natural ``zero section'' $Y\to \res_{R/k}\big(Y\times_kR\big)$ coming from the identity morphism on $Y\times_kR$.
Applying this to $Y=\jet^m\tau_n X$ and $R=\D_n(k)$ gives us a $k$-morphism
$\jet^m\tau_n X\to \res_{\D_n(k)/k}\big(\jet^m\tau_n X\times_k\D_n(k)\big)$.
Composing this with $\res_{\D_n(k)/k}\big(\jet^m(r)\big)$ yields a morphism
$$\phi^X_{m,n}:\jet^m\tau_n X\to\tau_n\jet^mX$$
of schemes over $k$.
This is the {\em interpolating map} of section~6 of~\cite{paperA}.\footnote{In fact the interpolating map was defined differently in~\cite{paperA} where we instead describe its action on points, but it is a straightforward exercise to see that the above description is an equivalent characterisation.}
It is a morphism of linear spaces over $\tau_n X$ and it satisfies the following properties:

\begin{proposition}
\label{phi-properties-hasse}
\begin{itemize}
\item[(a)]
The interpolating map is compatible with $\pi$.
That is, for all $n\geq n'$, the following diagram commutes:
$$\xymatrix{
\jet^m\tau_n X\ar[d]^{\phi_{m,n}}\ar[rr]^{\jet^m(\hat\pi_{n,n'})}&&\jet^m\tau_{n'} X\ar[d]^{\phi_{m,n'}}\\
\tau_n\jet^mX\ar[rr]_{\hat\pi_{n,n'}^{\jet^mX}}&&\tau_{n'}\jet^mX
}$$
\item[(b)]
The interpolating map is compatible with $\Delta$ in the sense that for all $n,n'$ the following diagram commutes:
$$\xymatrix{
\jet^m\tau_{n+n'} X\ar[d]^{\phi_{m,n+n'}}\ar[rr]^{\jet^m(\hat\Delta_{(n,n')})}&&\jet^m\tau_{n'}\tau_n X\ar[d]^{\phi_{m,\D_{(n,n')},E_{(n,n')}}}\\
\tau_{n+n'}\jet^mX\ar[rr]_{\hat\Delta^{\jet^mX}_{(n,n')}}&&\tau_{n'}\tau_n\jet^mX
}$$
\item[(c)]
The following diagram commutes for all $n, n'$
$$\xymatrix{
\jet^m\tau_{n'}\tau_n X\ar[dd]_{\phi_{m,\D_{(n,n')},E_{(n,n')}}}\ar[dr]^{\phi_{m,n'}^{\tau_n X}} &\\
& \tau_{n'}\jet^m\tau_n X\ar[dl]^{\tau_{n'}(\phi_{m,n}^X)}\\
\tau_{n'}\tau_n\jet^mX
}$$
\end{itemize}
\end{proposition}

\begin{proof}
Part~(a) is Proposition~6.4(c) of~\cite{paperA}, applied to $\alpha=\pi_{n,n'}$.
Part~(b) is Proposition~6.4(c) of~\cite{paperA}, applied to $\alpha=\Delta_{n,n'}$.
Part~(c) is Proposition~6.4(b) of~\cite{paperA}, applied to $(\mathcal{E},e)=(\D_n,E_n)$ and $(\mathcal{F},f)=(\D_{n'},E_{n'})$.
\end{proof}

Besides the above foundational properties of the interpolating map, the main observation from~\cite{paperA} is the following fact:

\begin{fact}[Corollary~6.8 of~\cite{paperA}]
\label{gensurjective}
Suppose $k$ is a field and $X$ is of finite-type.
If $p\in X(k)$ is smooth, then for all $m,n\in\mathbb{N}$, $\phi_{m,n}^X$ restricts to a surjective linear map between the fibres of $\jet^m\tau_n X$ and $\tau_n\jet^mX$ over $\nabla_n(p)\in\tau_n X(k)$.
\qed
\end{fact}

\section{Generalised Hasse-Schmidt subschemes}
\label{sec-hassescheme}

\noindent
Fix an iterative Hasse-Schmidt system $(\underline \D,\Delta)$ over $A$ and an iterative $\underline \D$-ring $(k,E)$.

It is possible to develop a theory of $\underline \D$-schemes in analogy with algebraic schemes starting with a theory of sheaves of $\underline \D$-rings.
This would generalise, for example, the approach taken by Kovacic in~\cite{kovacic02} and Benoist in~\cite{benoist08} in differential-algebraic geometry.
As their work exhibits, there are a number of subtle and interesting foundational problems that arise in doing so.
Moreover, for the $\underline\D$-jet space theory we wish to develop here, it seems essential that our $\underline\D$-schemes come equipped with a fixed embedding in an algebraic scheme.
So we will restrict ourselves to the following approach:
we fix an algebraic scheme $X$ over $k$ and introduce only a theory of Hasse-Schmidt subschemes of $X$.

\begin{definition}[Generalised Hasse-Schmidt subschemes]
\label{hassevariety}
Suppose $X$ is a scheme over $k$.
A {\em $\underline\D$-subscheme of $X$ over $k$} is a collection of closed subschemes over $k$,
$$\underline{Z}=\big(Z_n\subseteq\tau_n X:n\in\mathbb{N}\big)$$
such that:
\begin{itemize}
\item[(1)]
For all $n\in\mathbb{N}$, the structure morphism $\hat\pi_{n}: \tau_{n+1} X\to\tau_n X$ restricts to a morphism from $Z_{n+1}$ to $Z_n$.
\item[(2)]
For all $m\in\mathbb{N}$, the morphism $\hat\Delta_{(m,1)}:\tau_{m+1} X\to\tau(\tau_m X)$ induced by iterativity, restricts to a morphism from $Z_{m+1}$ to $\tau(Z_m)$.
\end{itemize}
By the {\em $k$-rational points of $\underline Z$} we mean the subset of $X(k)$ given by
$$\underline Z(k):=\big\{p\in X(k):\nabla_n(p)\in Z_n(k), \text{ for all }n\in\mathbb{N}\big\}$$
We will also utilise the following terminology:
\begin{itemize}
 \item
$\underline{Z}$ is {\em dominant} if each projection $Z_{n+1}\to Z_n$ is dominant.
\item
$\underline{Z}$ is {\em separable} if each projection $Z_{n+1}\to Z_n$ is separable.
\item
$\underline{Z}$ is {\em irreducible} if each $Z_n$ is irreducible.
\item
If $k$ is a field and each $Z_n$ is reduced and of finite-type over $k$, then we say that $\underline{Z}$ is a Hasse-Schmidt {\em subvariety} of $X$.
\end{itemize}
\end{definition}

Note that every closed subscheme $Y\subseteq X$ can be viewed as a dominant Hasse-Schmidt subscheme by considering $\underline Y:= \big(\tau_n Y:n\in\mathbb{N}\big)$; so that $\underline Y(k)=Y(k)$. (Domination is by Proposition~\ref{liftinglemma}.)

We now establish a few lemmas which clarify the definitions.

\begin{lemma}
\label{n=1enough}
Suppose $\underline Z$ is a $\underline\D$-subscheme of a scheme $X$ over $k$.
For all $m,n\in\mathbb{N}$, the morphism $\hat\Delta_{(m,n)}:\tau_{m+n} X\to\tau_n\tau_m X$ induced by iterativity, restricts to a morphism from $Z_{m+n}$ to $\tau_n(Z_m)$.
In particular, $Z_n\subseteq \tau_n(Z_0)$ for each $n\in\mathbb{N}$.
\end{lemma}

\begin{proof}
Note that part~(2) of Definition~\ref{hassevariety} is just the $n=1$ case of this lemma.
The `in particular' clause follows from the $m=0$ case of the main clause using the fact that $\Delta_{(0,n)}=\id$.

We prove the lemma by induction on $n$, the case of $n=0$ being trivially true as $\hat\Delta_{(m,0)}=\id$.
Now, for any $n$, consider the following diagram which commutes by the associativity of $\Delta$ (part (b) of Definition~\ref{iterative-hasse-system}):
$$\xymatrix{
\tau_{m+n+1}X\ar[rr]^{\hat\Delta_{m,n+1}^X} \ar[d]_{\hat\Delta_{m+n,1}^X} && \tau_{n+1}\tau_mX\ar[d]^{\hat\Delta_{n,1}^{\tau_mX}}\\
\tau\tau_{m+n}X\ar[rr]^{\tau(\hat\Delta_{m,n}^X)} && \tau\tau_n\tau_mX
}$$
Let us chase $Z_{m+n+1}$ from the top left to the bottom right, counter-clockwise.
By part~(2) of Definition~\ref{hassevariety}, $\hat\Delta_{m+n,1}$ takes $Z_{m+n+1}$ to $\tau(Z_{m+n})$.
By the induction hypothesis, $\hat\Delta_{(m,n)}$ takes $Z_{m+n}$ to $\tau_n(Z_m)$.
Applying the functor $\tau$ we get that $\tau(\hat\Delta_{m,n})$ takes $\tau(Z_{m+n})$ to $\tau\tau_nZ_m$.
So, counter-clockwise, $Z_{m+n+1}$ gets sent to $\tau\tau_nZ_m$.
So, from the above diagram, we get that $\hat\Delta_{m,n+1}$ restricts to a morphism from $Z_{m+n+1}$ to $\big(\hat\Delta_{n,1}^{\tau_mX}\big)^{-1}(\tau\tau_nZ_m)$.
Now, as $\hat\Delta_{n,1}^{\tau_mX}$ is a closed embedding (Remark~\ref{deltainduced}), and $\hat\Delta_{n,1}^{\tau_mX}\upharpoonright \tau_{n+1}Z_m=\hat\Delta_{n,1}^{Z_m}:\tau_{n+1}Z_m\to\tau\tau_nZ_m$ (this is the functoriality of $\hat \Delta$, cf. Proposition~4.8(b) of~\cite{paperA}), we get that $\big(\hat\Delta_{n,1}^{\tau_mX}\big)^{-1}(\tau\tau_nZ_m)=\tau_{n+1}Z_m$.
So $\hat\Delta_{m,n+1}$ restricts to a morphism from $Z_{m+n+1}$ to $\tau_{n+1}Z_m$, as desired.
\end{proof}

Hasse-Schmidt subschemes, as we have defined them, are given by a {\em compatible} sequence of algebraic conditions on the prolongation spaces.
It might be more natural to consider arbitrary algebraic conditions.
The following lemma describes how to produce an iterative
$\underline{\mathcal D}$-subscheme from a given system of $\underline\D$-equations.

\begin{lemma}
\label{makehasse}
Suppose $X$ is a scheme over $k$ and $Y_n\subseteq \tau_n X$ is a sequence of closed subschemes.
Then there exists a dominant Hasse-Schmidt subscheme $\underline Z=(Z_n)$ such that for any $\underline\D$-ring $k'$ extending $k$,
$$\underline Z(k')=\{p\in X(k'):\nabla_n(p)\in Y_n(k'), n<\omega\}.$$
\end{lemma}

\begin{proof}
Let $C$ be the set of all sequences of closed subschemes $W_n\subseteq \tau_n X$ such that for all $\underline\D$-ring extensions $k'$ of $k$,
$$\{p\in X(k'):\nabla_n(p)\in W_n(k'), n<\omega\}=\{p\in X(k'):\nabla_n(p)\in Y_n(k'), n<\omega\}.$$
Note that if $(W_n)$ and $(W_n')$ are in $C$ then so is $(W_n\cap W_n')$.
So $C$ has a least element $(Z_n)$.
We claim that $\underline Z:=(Z_n)$ is a dominant Hasse-Schmidt subscheme of $X$.

Fixing $m$ we show that $\hat\pi_{m+1,m}$ restricts to a map from $Z_{m+1}$ to $Z_m$.
Indeed let $(W_n)$ be defined by $W_n:=Z_n$ for $n\neq m+1$ and $W_{m+1}:=\hat\pi^{-1}_{m+1,m}(Z_m)\cap Z_{m+1}$.
Then,
since by Lemma~\ref{nablaprop} $\hat\pi\big(\nabla_{m+1}(p)\big)=\nabla_m(p)$,
$(W_n)$ is again in $C$ and $(W_n)\subseteq (Z_n)$.
By minimality we have $W_m=Z_m$.
So $\hat\pi_{m+1,m}$ restricts to a map from $Z_{m+1}$ to $Z_m$, as desired.

Next, we show that $\hat\pi_{m+1,m}$ restricts to a dominant map from $Z_{m+1}$ to $Z_m$.
Let $W_m:=\overline{\hat\pi_{m+1,m}(Z_{m+1})}$ and $W_n:=Z_n$ for all $n\neq m$.
Again Lemma~\ref{nablaprop} implies that $(W_n)$ is in $C$ and so by minimality $W_m=Z_m$, as desired.

Finally, fixing $m$ we need to show that $\hat\Delta_{(m,1)}$ restricts to a morphism from $Z_{m+1}$ to $\tau(Z_m)$.
Define $(W_n)$ so that $W_n=Z_n$ for each $n\neq m+1$ and $W_{m+1}:=Z_{m+1}\cap\hat\Delta_{(m,1)}^{-1}\big(\tau(Z_m)\big)$.
Since $\hat\Delta_{(m,1)}\big(\nabla_{m+1}(p)\big)=\nabla\big(\nabla_m(p)\big)$ by Remark~\ref{deltainduced}, $(W_n)$ is in $C$.
By minimality we get $W_{m+1}=Z_{m+1}$, which means that $\hat\Delta_{(m,1)}$ restricts to a morphism from $Z_{m+1}$ to $\tau(Z_m)$, as desired.
\end{proof}

\begin{example}[Kolchin closed sets]
\label{makehassedifferential}
Consider the iterative Hasse-Schmidt system $\operatorname{HSD}_e$ of Example~\ref{differential} and Proposition~\ref{iterativedifferential}.
If $(k,\partial_1,\dots,\partial_e)$ is a (partial) differential field of characteristic zero (viewed in the natural way as an iterative $\operatorname{HSD}_e$-field) then every system of differential-polynomial equations over $k$, in say $\ell$ differential variables, gives rise to a dominant $\operatorname{HSD}_e$-subscheme of $\mathbb{A}^\ell$ whose $k$-points are exactly the solutions to the system in $k^\ell$.
Indeed, such differential-polynomial equations correspond to algebraic conditions on the prolongation spaces -- and thus give rise to a system of closed subschemes $Y_n\subseteq\tau_n(\mathbb{A}^\ell)$.
Now apply Lemma~\ref{makehasse}.
So our $\underline{\mathcal D}$-subschemes generalise Kolchin closed sets from differential algebra.
\end{example}

Before moving on, let us briefly discuss the issue of irreducibility for Hasse-Schmidt subvarieties.
The definition we have given, namely that each $Z_n$ is irreducible, is rather strong.
For example, one cannot expect that every Hasse-Schmidt subvariety can be written as a finite union of irreducible Hasse-Schmidt subvarieties.
However, we do have the following:

\begin{lemma}
\label{preirred}
Suppose $k$ is a field and $X$ is a variety (so reduced and of finite-type) over $k$, and $\underline Z$ is a dominant Hasse-Schmidt subvariety of $X$ over $k$.
Then for each $N<\omega$ there exists finitely many dominant Hasse-Schmidt subvarieties $\underline Y^1,\dots,\underline Y^\ell\subseteq \underline Z$ such that
\begin{itemize}
\item
for any $\underline\D$-ring $k'$ extending $k$, $\underline Z(k')=\underline Y^1(k')\cup\dots\cup \underline Y^\ell(k')$, and
\item
for all $m\leq N$, $Y^i_m$ is $k$-irreducible for $i=1,\dots,\ell$.
\end{itemize}
\end{lemma}

\begin{proof}
The proof is by Noetherian induction on $Z_N$.
If $Z_N$ is $k$-irreducible then so are all the $Z_m$ for $m\leq N$ by dominance -- and hence we are done.
Suppose we can decompose $Z_N$ as a union of two proper $k$-closed sets, say $U$ and $V$.
Replacing $Z_N$ by $U$  and then by $V$ in the sequence $(Z_n)$, and then applying Lemma~\ref{makehasse} to the these two sequences, we get dominant Hasse-Schmidt subvarieties $\underline Z^U$ and $\underline Z^V$ over $k$ such that $\underline Z(k')=\underline Z^U(k')\cup\underline Z^V(k')$ for any $\underline\D$-ring extension $k'$ of $k$.
Now $Z^U_N\subseteq U\subsetneq Z_N$ and $Z^V_N\subseteq V\subsetneq Z_N$.
By induction there exists $\underline Y^1,\dots,\underline Y^\ell$ satisfying the lemma for $\underline Z^U$, and $\underline W^1,\dots,\underline W^s$ satisfying the lemma for $\underline Z^V$.
But then $\{\underline Y^i,\underline W^j:i=1,\dots,\ell,j=1,\dots,s\}$ works for $\underline Z$.
\end{proof}

Iterating the above construction, every dominant Hasse-Schmidt subvariety can be written as a union of $2^{\aleph_0}$-many $k$-irreducible Hasse-Schmidt subvarieties.

\subsection{Some $\mathcal D$-algebra}
\label{dalgebra}
While it is our intention to avoid developing the algebraic side of this theory in detail, we will now present the
``$\underline\D$-co-ordinate ring'' associated to an (affine) $\underline\D$-subscheme.

The $\underline\D$-co-ordinate ring that we will define will be an iterative $\underline\D$-ring {\em extension} of $(k,E)$; that is, a $k$-algebra $a:k\to R$ with an iterative $\underline\D$-ring structure $E'$ on $R$ such that the following diagram commutes:
$$\xymatrix{
R\ar[rr]^{E_n'} && \D_n(R)\\
k\ar[u]^{a}\ar[rr]^{E_n} &&\D_n(k)\ar[u]_{\D_n(a)}
}$$
for all $n\in\mathbb N$.
Equivalently, each $E'_n:R\to\D_n^{E_n}(R)$ will be a $k$-algebra homomorphism.
Given a $k$-algebra $R$, the following proposition explains how to recognise an iterative $\underline\D$-ring structure $E'$ on $R$ extending $(k,E)$, in terms of the morphism on spectra induced by $E'$.\footnote{We are grateful to the referee for suggesting this formulation, which substantially clarifies our original exposition.}

\begin{proposition}
\label{refsidea}
Suppose $R$ is a $k$-algebra, $Y:=\spec(R)$, and we are given a sequence of morphisms over $R$, $u=(u_n:Y\to\tau_n Y \ | \ n\in\mathbb N)$, satisfying the following properties:
\begin{itemize}
\item[(i)]
$u_0=\id$
\item[(ii)]
$u$ is compatible with $\pi$.
That is, for all $m\geq n\in\mathbb N$, the following commutes:
$$\xymatrix{
\tau_mY\ar[rr]^{\hat\pi_{m,n}} && \tau_nY\\
& Y\ar[ul]^{u_m}\ar[ur]_{u_n}
}$$
\item[(iii)]
$u$ is compatible with $\Delta$.
That is, for all $m,n\in\mathbb N$, the following commutes:
$$\xymatrix{
\tau_{m+n}Y\ar[rr]^{\hat\Delta_{(m,n)}} && \tau_n\tau_mY\\
& Y\ar[ul]^{u_{m+n}}\ar[ur]_{\tau_n(u_m)\circ u_n}
}$$
\end{itemize}
Note that $u_n\in\tau_n Y(R)$ for each $n\in \mathbb N$.
Let $\tilde u_n$ be the corresponding $\D_n^{E_n}(R)$-point of $Y$, and let $E_n':R\to\D_n^{E_n}(R)$ be the corresponding $k$-algebra homomorphism.
Then, $E':=(E_n' \ | \ n\in\mathbb N)$ makes $R$ into an iterative $\underline\D$-ring extension of $(k,E)$.
\end{proposition}

\begin{proof}
First of all, let us recall that the identification $\tau_n Y(R)=Y\big(\D_n^{E_n}(R)\big)$ is by $p\mapsto r_n^Y\circ\big(p\times_k\D_n(k)\big)$, where $r_n^Y:\tau_nY\times_k\D_n(k)\to Y$ is the canonical morphism associated to a prolongation, viewed as a morphism over $k$ in the appropriate way.
See the discussion on page~\pageref{canonicalmap}.

The fact that $E_n'$ is a $k$-algebra homomorphism from $R$ to $\D_n^{E_n}(R)$ says exactly that it is an $A$-algebra homomorphism from $R$ to $\D_n(R)$ extending $E_n:k\to\D_n(k)$.
So it suffices to check that $(R,E')$ is an iterative $\underline\D$-ring.
That $(R,E')$ is a $\underline\D$-ring is more or less immediate from conditions~(i) and ~(ii) on $u$.
That condition~(iii) implies iterativity of $E'$ requires a little more work; namely, we need to show that the $\D_m^{E_m}\big(\D_n^{E_n}(R)\big)$-point of $Y$ corresponding to $\tau_n(u_m)\circ u_n$ agrees with $E_{(m,n)}'=E_m'\circ \D_m(E'_n):R\to\D_m\big(\D_n(R)\big)$.
In terms of $u$, what we need to prove is the commuting of the following diagram:
\begin{eqnarray}
\label{refsidea-iterate}
\xymatrix{
Y\times_k\D_m(k) \ar[rr]^{u_m\times_k\D_m(k)} &&
\tau_mY\times_k\D_m(k) \ar[drr]^{r^Y_m}\\
\tau_nY\times_k\D_m\D_n(k)\ \  \ar[rr]_{\tau_n(u_m)\times_k\D_m\D_n(k)} \ar[u]^{r^Y_n\times_k\D_m(k)} &&
\ \ \tau_n\tau_mY\times_k\D_m\D_n(k) \ar[rr]_{\ \ \ \ \ \ \ \ \ \ r^Y_{(n,m)}} &&
Y
}
\end{eqnarray}
Here, strictly speaking, $r_n^Y\times_k\D_m(k)$ is the isomorphism from $\tau_n Y\times_k\D_m\D_n(k)$ to $\tau_n Y\times_k\D_n(k)\times_k\D_m(k)$ followed by $r_n^Y$ base changed up from $k$ to $\D_m(k)$.

Now, Lemma~4.14 of~\cite{paperA} tells us how the canonical morphisms compose, allowing us to reduce the commuting of~(\ref{refsidea-iterate}) to that of
$$\xymatrix{
Y\times_k\D_m(k) \ar[rr]^{u_m\times_k\D_m(k)} &&
\tau_mY\times_k\D_m(k)\\
\tau_nY\times_k\D_m\D_n(k)\ \  \ar[rr]_{\tau_n(u_m)\times_k\D_m\D_n(k)} \ar[u]^{r^Y_n\times_k\D_m(k)} &&
\ \ \tau_n\tau_mY\times_k\D_m\D_n(k) \ar[u]_{r_n^{\tau_m Y}\times_k\D_m(k)}
}$$
Dropping the final base change to $\D_m(k)$ everywhere, we further reduce the problem to verifying the commutativity of
$$\xymatrix{
Y \ar[rr]^{u_m} &&
\tau_mY\\
\tau_nY\times_k\D_n(k)\ \  \ar[rr]_{\tau_n(u_m)\times_k\D_n(k)} \ar[u]^{r^Y_n} &&
\ \ \tau_n\tau_mY\times_k\D_n(k) \ar[u]_{r_n^{\tau_m Y}}
}$$
which is just the functoriality of the prolongations and the associated canonical morphisms.
\end{proof}

\begin{remark}
We pointed out in Remark~\ref{adjoint} that the functor on $k$-algebras that takes the co-ordinate ring of an affine $k$-scheme to the co-ordinate ring of its $n$th prolongation is left-adjoint to $\D_n^{E_n}$.
For any $k$-algebra $R$, the adjoint transpositions witnessing the above fact take $\hat\Delta^R_{(m,n)}$ to the $R$-dual of $\Delta_{(m,n)}^R$.
From this point of view, Proposition~\ref{refsidea} becomes the adjoint version of the definition of a $\Delta$-iterative $\underline\D$-ring (Definition~\ref{iterative-hasse-system}).
In particular, while we will not make use of it, the converse of Proposition~\ref{refsidea} is also true and similarly verified: all iterative $\underline\D$-ring structures on $R$ extending $(k,E)$ are obtained from such a sequence $u$.
\end{remark}

Suppose $X$ is an affine scheme over $(k,E)$.
Let $k\langle X\rangle$ denote the direct limit of the co-ordinate rings $k[\tau_mX]$ under the homomorphisms induced by $\hat \pi_{m+1,m}:\tau_{m+1}X\to\tau_mX$.
Proposition~\ref{refsidea} gives us a natural $\underline\D$-ring structure on $k\langle X\rangle$:
Let $Y:=\spec\big(k\langle X\rangle\big)$ and set $u_n:Y\to\tau_nY$ to be the inverse limit of the morphism $\hat\Delta_{(m,n)}:\tau_{m+n}X\to\tau_n\tau_m X$ as $m$ goes to infinity.
The defining properties of $\Delta$ ensure that the hypotheses of Proposition~\ref{refsidea} are satisfied by the sequence $u=(u_n \ | \ n\in\mathbb N)$.
The main point is that $\Delta$ is compatible with itself; the associativity of $\Delta$ expressed by part~(b) of  Definition~\ref{iterative-hasse-system} implies condition~(iii) is true of the sequence $u$.
We thus obtain an iterative $\underline\D$-ring $(k\langle X\rangle,E^X)$ extending $(k,E)$.
This is the {\em $\underline\D$-co-ordinate ring} of the affine scheme $X$.

Now let $\underline Z$ be a $\underline\D$-subscheme of $X$ over $k$, and denote by $k\langle \underline Z\rangle$ be the direct limit of the $k[Z_m]$.
Then since $\hat\Delta_{m,n}^X$ restricts to a morphism from $Z_{m+n}$ to $\tau_n Z_m$ (by definition of $\underline\D$-subscheme), taking inverse limits shows that $u_n$ restricts to a morphism from $\spec\big(k\langle\underline Z\rangle\big)\to\tau_n\spec\big(k\langle\underline Z\rangle\big)$ also satisfying the hypotheses of Proposition~\ref{refsidea}.
Hence $E^X$ induces an iterative $\underline\D$-ring structure on $k\langle\underline Z\rangle$ extending $(k,E)$.
We denote the $\underline\D$-ring structure by $E^{\underline Z}$, and call $\big(k\langle\underline Z\rangle,E^{\underline Z}\big)$ the {\em $\underline\D$-co-ordinate ring of $\underline Z$}.
Note that the quotient map $\rho:k\langle X\rangle\to k\langle\underline Z\rangle$ is a {\em $\underline\D$-homomorphism}; that is,
$$
\xymatrix{
k\langle X\rangle\ar[d]^\rho\ar[r]^{E_n^X} & \D_n\big(k\langle X\rangle\big)\ar[d]^{\D_n(\rho)}\\
k\langle\underline Z\rangle\ar[r]^{E_n^{\underline Z}} & \D_n\big(k\langle \underline Z\rangle\big)\\
}
$$
commutes for each $n$.

\begin{remark}
\label{oldidea}
The maps $E^X_n:k\langle X\rangle\to\D_n\big(k\langle X\rangle\big)$ can also be seen as the direct limit as $m$ goes to infinity of certain $k$-algebra homomorphisms
$$E^{X,m}_n:k[\tau_mX]\to\D_n^{E_n}\big(k[\tau_{m+n}X]\big)$$
The $E^{X,m}_n$ are the maps on co-ordinate rings associated to the morphisms
$$r_n^{\tau_m X}\circ\big(\hat\Delta_{m,n}\times_k\D_n(k)\big) \ : \ 
\tau_{m+n} X\times_k\D_n(k)\to\tau_m X$$
Similarly, $E^{\underline Z}_n$ on $k\langle\underline Z\rangle$ is the direct limit as $m$ goes to infinity of the $k$-algebra homomorphisms
$$E_n^{\underline Z,m}:k[Z_m]\to\D_n^{E_n}\big(k[Z_{m+n}]\big)$$
induced by
$r_n^{\tau_mX}\circ \big(\hat\Delta_{(m,n)}\times_k\D_n(k)\big)$ restricted to $Z_{m+n}\times_k\D_n(k)$.
\end{remark}

\subsection{Generic points in fields}
Let us now specialise to the case when $k$ is a field.
The geometric theory of $\underline \D$-subvarieties over $k$ goes much more smoothly if we can be guaranteed that in some $\underline \D$-field extension our $\underline \D$-subvariety has a generic point.
This will not always be the case.
In this section we introduce a condition on the Hasse-Schmidt system which will guarantee us the existence of generic points in $\underline\D$-field extensions.

\begin{definition}
Suppose $(R,E)$ is an iterative $\underline\D$-ring and $S\subseteq R\setminus\{0\}$ is a multiplicatively closed set. We say that {\em $E$ localises to $S^{-1}R$} if for every $n\in\mathbb{N}$, $E_n:R\to\D_n(R)$ extends to a ring homomorphism $\tilde{E}_n$ such that
$$
\xymatrix{
R\ar[rr]^{E_n}\ar[d]_{\iota} && \D_n(R)\ar[d]^{\D_n(\iota)}\\
S^{-1}R\ar[rr]^{\tilde{E}_n}&&\D_n(S^{-1}R)
}$$
commutes.
We say that the Hasse-Schmidt system $\underline\D$ {\em extends to fields} if whenever $(R,E)$ is an iterative $\underline \D$-integral domain, and $K$ is the fraction field, then $E$ localises to $K$.
\end{definition}

\begin{remark}
\label{localcompat}
{\em Suppose $E$ localises to $S^{-1}R$.
Then each $\tilde{E}_n$ is uniquely determined and $\big(S^{-1}R,\tilde{E}:=(\tilde{E}_n:n\in\mathbb{N})\big)$ is an iterative $\D$-ring.}
Indeed, if $\frac{\iota a}{\iota b}\in S^{-1}R$ then as $\tilde{E}_n$ is a ring homomorphism, $\tilde{E}_n(\iota b)$ is a unit in $\D_n(S^{-1}R)$ and $\tilde E_n(\frac{\iota a}{\iota b})=\frac{E_n(\iota a)}{E_n(\iota b)}=\frac{\D_n(\iota)\big(E_n(a)\big)}{\D_n(\iota)\big(E_n(b)\big)}$.
This gives uniqueness.
Two straighforward diagram chases now show that $\tilde{E}:=(\tilde{E}_n:n\in\mathbb{N})$ is compatible with $\pi$ and $\Delta$ (since $E$ is), hence making $S^{-1}R$ into an iterative $\underline\D$-ring.
\end{remark}

\begin{proposition}
\label{nillocal}
Suppose that for all $n\in\mathbb N$, the kernel of $\pi_{n,0}:\D_n(A)\to A$ is a nilpotent ideal.
Then for any iterative $\underline \D$-ring $(R,E)$ over $A$, and any multiplicatively closed set $S\subseteq R\setminus\{0\}$, $E$ localises to $S^{-1}R$.
In particular, $\underline\D$ extends to fields.
\end{proposition}

\begin{proof}
Fix $n\in\mathbb{N}$ and let $I_n\subseteq\D_n(A)$ be the kernel of $\pi_{n,0}$.

By the universal property of localisations, the existence of such $\tilde E_n$ will follow once we show that $\D_n(\iota)\big(E_n(s)\big)$ is a unit in $\D_n(S^{-1}R)$, for each $s\in S$.
Consider the commuting square
$$
\xymatrix{
R\ar[d]_{\iota} && \D_n(R)\ar[ll]_{\pi_{n,0}^R}\ar[d]^{\D_n(\iota)}\\
S^{-1}R&&\D_n(S^{-1}R)\ar[ll]_{\pi_{n,0}^{S^{-1}R}}
}$$
Now, the kernel of the surjective homomorphism $\pi_{n,0}:\D_n(S^{-1}R)\to S^{-1}R$ is the nilpotent ideal $S^{-1}R\otimes_A I_n$, and hence the units of $\D_n(S^{-1}R)$ are just the pull-backs of the units of $S^{-1}R$.
In particular, as $\pi_{n,0}\big(\D_n(\iota)(E_n(s))\big)=\iota(s)$ is a unit in $S^{-1}R$,
we get that $\D_n(\iota)\big(E_n(s)\big)$ is a unit in $\D_n(S^{-1}R)$, as desired.
So the required extensions $\tilde{E}_n:S^{-1}R\to\D_n(S^{-1}R)$ exist.
\end{proof}

\begin{corollary}
\label{extendtofielddifferential}
The iterative Hasse-Schmidt system $\operatorname{HSD}_e$ used to study Hasse-Schmidt differential rings in Example~\ref{differential} extends to fields.
\qed
\end{corollary}

Let us return to our study of Hasse-Schmidt subvarieties over the iterative $\underline\D$-field $(k,E)$.
The following proposition ensures that if $\underline\D$ extends to fields then Hasse-Schmidt varieties over $k$ will always have many rational points in $\underline\D$-field extensions of $(k,E)$.
More precisely,

\begin{proposition}
\label{generics}
Suppose $\underline\D$ extends to fields.
Let $X$ be an algebraic variety over $k$ and $\underline Z$ a dominant irreducible Hasse-Schmidt subvariety of $X$ over $k$.
Then there exists an iterative $\underline \cD$-field extension $(K,E)$ of $(k,E)$ and a point $b\in \underline Z(K)$ such that $\nabla_n(b)$ is $k$-generic in $Z_n$ for all $n\in\mathbb{N}$.
We say that $b$ is {\em $k$-generic in $\underline Z$}.
\end{proposition}

\begin{proof}
Suppose $\underline Z$ is a dominant irreducible Hasse-Schmidt subvariety of $X$ over $k$.
We construct a $\underline\D$-extension $K$ of $k$ such that $\underline Z(K)$ contains a ``$k$-generic'' point.
By irrreducibility each $k[Z_n]$, and hence the $\underline\D$-co-ordinate ring $k\langle\underline Z\rangle$, is an integral domain.
Let $K$ be the fraction field of $k\langle\underline Z\rangle$, the {\em Hasse-Schmidt rational function field} of $\underline Z$.
Since $\underline\D$ extends to fields, the iterative $\underline\D$-ring structure on $k\langle\underline Z\rangle$ extends to an iterative $\underline\D$-field structure on $K$.

Let $f:k\langle\underline Z\rangle\to K$ be the inclusion of the integral domain in its fraction field.
For each $n\in \mathbb{N}$, let $f_n:k[Z_n]\to K$ be the homomorphism obtained from $a$ by precomposing with the direct limit map from $k[Z_n]$ to $k\langle\underline Z\rangle$.
The dominance of the maps $Z_{m+1}\to Z_m$ imply that $f_n$ factors through $k(Z_n)$, the rational function field of $Z_n$.
That is, the point $a_n=\spec(f_n)\in Z_n(K)$ is $k$-generic in $Z_n$.

We prove that for each $n\in\mathbb{N}$,
$\nabla_n(a_0)=a_n$.
This will suffice as it implies that $a_0$ is the $k$-generic point of $\underline Z(K)$ that we want.
Under the standard identification, $\nabla_n(a_0)$ is a $\D_n^{E_n}(K)$ point of $Z_0$.
As $Z_n\subseteq\tau_nZ_0$, we can also view $a_n$ as a $\D_n^{E_n}(K)$ point of $Z_0$.
It is as such that we prove they agree.

From the fact that $E_n$ on $K$ extends $E_n$ on $k\langle\underline Z\rangle$, and taking spectra, we get

$$
\xymatrix{
\spec(K)\ar[d]_{\spec(f)} && \spec(K)\times_k\D_n(k) \ar[d]^{\spec\big(\D_n(f)\big)}\ar[ll]_{\spec(E_n)}\\
\spec\big(k\langle\underline Z\rangle\big) && \spec(k\langle\underline Z\rangle)\times_k\D_n(k) \ar[ll]_{\spec(E_n)} 
}$$
On the other hand we know that  $E_n$ on $k\langle\underline Z\rangle$ is the direct limit as $m$ goes to infinity of the $k$-algebra homomorphisms
$E_n^{\underline Z,m}:k[Z_m]\to\D_n^{E_n}\big(k[Z_{m+n}]\big)$
induced by
$r_n^{\tau_mX}\circ \big(\hat\Delta_{(m,n)}\times_k\D_n(k)\big):\tau_{m+n}X\times_k\D_n(k)\to\tau_mX$ restricted to $Z_{m+n}\times_k\D_n(k)$ -- see Remark~\ref{oldidea}.
It follows, setting $m=0$, that
$$
\xymatrix{
\spec\big(k\langle\underline Z\rangle\big)\ar[d] &&\spec(k\langle\underline Z\rangle)\times_k\D_n(k)\ar[ll]_{\spec(E_n)}\ar[d]\\
Z_0 &&Z_n\times_k\D_n(k)\ar[ll]_{r^X_n|_{Z_n\times_k\D_n(k)}}
}$$
commutes.
Note that $r^X_n|_{Z_n\times_k\D_n(k)}$ is the restriction of $r_n^{Z_0}:\tau_nZ_0\times_k\D_n(k)\to Z_0$ to $Z_n\times_k\D_n(k)$.
Putting the two diagrams together we get
$$
\xymatrix{
\spec(K)\ar[dd]_{a_0} &&\spec(K)\times_k\D_n(k) \ar[dd]^{a_n\otimes_k\D_n(k)} \ar[ll]_{\spec(E_n)}\\
\\
Z_0 & \tau_nZ_0\times_k\D_n(k) \ar[l]_{r_n^{Z_0}} & Z_n\times_k\D_n(k)\ar[l]_{\supseteq}
}$$
Top-right to bottom-left counter-clockwise is exactly $\nabla(a_0)$ viewed as a $\D_n^{E_n}(K)$-point of $Z_0$, whereas top-right to bottom-left clockwise is $a_n$ viewed as a $\D_n^{E_n}(K)$-point of $Z_0$
Hence $\nabla_n(a_0)=a_n$.

We have shown that $a_0\in\underline Z(K)$ and that it is $k$-generic in $\underline Z$.
\end{proof}

\begin{definition}
\label{rich}
We say that $(k,E)$ is {\em rich} if whenever $X$ is an algebraic variety over $k$ and $\underline Z$ is a dominant irreducible Hasse-Schmidt subvariety of $X$ over $k$, then $\nabla_n\big(\underline Z(k)\big)$ is Zariski-dense in $Z_n$ for all $n\in\mathbb{N}$.
\end{definition}

\begin{corollary}
\label{densepoints}
Suppose $\underline\D$ extends to fields.
Then every iterative $\underline\D$-field extends to a rich iterative $\underline\cD$-field.
\end{corollary}

\begin{proof}
Suppose $(k,E)$ is an iterative $\underline\D$-field.
We build a rich $\underline \cD$-field, $L$, as a direct limit of an $\omega_1$-chain of $\underline\cD$-field extensions of $k$.  
Start with $L_0 = k$.  
Given $L_m$, list all of the dominant irreducible $\underline\cD$-varieties over $L_m$,
$(\underline Z_\alpha:\alpha<\kappa)$.
We build $L_{m+1,\beta}$ by transfinite
recursion on $\beta<\kappa$.
At stage $\beta$, if $\beta>0$ then let $M = M_{m,\beta}$ be the union of
$L_{m+1,\gamma}$ for $\gamma < \beta$.
If $\beta=0$ then let $M=L_m$.
Let $\underline Z$ be an $M$-irreducible component of
$Z_\beta$.
Let $L_{m+1,\beta}\supseteq M$ and $a = a_{m,\beta}\in\underline Z_\beta(L_{m+1,\beta})$
be an $M$-generic point of $\underline Z$, as given by Proposition~\ref{generics}.
We then let $L_{m+1}$ be the union of the $L_{m+1,\alpha}$, $\alpha<\kappa$.  
At limit stages we take unions, and we set $L:= L_{\omega_1}.$

Suppose now that $\underline Z = (Z_n)$ is a dominant irreducible Hasse-Schmidt variety over $L$ and $W$ a proper subvariety of some $Z_n$.
Then $Z$ and $W$  are defined over some countable subfield of $L$ and as
such are defined over (and irreducible over) some $L_m$.
So, for some $\beta$, $\underline Z = \underline Z_\beta$ for the listing of the dominant irreducible Hasse-Schmidt varieties over $L_m$.
We have $L_m\subseteq M_{m,\beta}\subseteq L_{m+1,\beta}\subseteq L_{m+1}\subseteq L$.
As $\underline Z$ is irreducible over $L$, it was already irreducible over
$M_{m,\beta}$.
Thus, $a_{m,\beta}\in\underline Z(L_{m+1,\beta})$ is $M_{m,\beta}$-generic, and hence $L_m$-generic.
In particular,
$\nabla_n(a_{m,\beta})$  is not an  element of  $W$.
Thus,  $\nabla_n(\underline Z(L))$
is not  contained in  $W(L)$.
We have shown that $\nabla_n\big(\underline Z(L)\big)$ is Zariski-dense in $Z_n$, for all $n\in\mathbb{N}$.
\end{proof}

We make immediate use of the existence of sufficiently many rational points in the following proposition, which we will need later, and which says that  applying $\nabla$ to the rational points of a Hasse-Schmidt subvariety produces the rational points of another Hasse-Schmidt subvariety.

\begin{proposition}
\label{nablanz}
Suppose $(k,E)$ is a rich $\underline D$-field, $X$ is a variety over $k$, and $\underline Z$ is a dominant irreducible Hasse-Schmidt subvariety of $X$ over $k$.
Then for each $m\in\mathbb{N}$ there exists a dominant Hasse-Schmidt subvariety $\underline Y$ of $Z_m$ with $\nabla_m\big(\underline Z(k)\big)=\underline Y(k)$.
We denote this Hasse-Schmidt subvariety by $\nabla_m\underline Z$.
\end{proposition}

\begin{proof}
There is an obvious candidate for $\underline Y$: set $\underline Y=(Y_n)$ where $Y_n$ is the image of $Z_{m+n}$ in $\tau_n(Z_m)$ under $\hat\Delta_{(m,n)}$.
Since $\hat\Delta_{(m,n)}$ is a closed embedding $Y_n$ is a closed subvariety of $\tau_n(Z_m)$.

We first show that $\underline Y$ is a dominant Hasse-Schmidt subvariety.
For the first condition we need to check that $\tau_{n+1}(Z_m)\to\tau_n(Z_m)$ induces a dominant map from $Y_{n+1}$ to $Y_n$.
But this follows from the fact that $\tau_{m+n+1} X\to\tau_{m+n} X$ induces a dominant map from $Z_{m+n+1}$ to $Z_{m+n}$, and from the compatibility of $\Delta$ with $\pi$ (cf. the commuting diagram in Definition~\ref{iterative-hasse-system}(a)).
The second condition requires us to confirm that $\hat \Delta_{(n,1)}^{Z_m}:\tau_{n+1}(Z_m)\to \tau\big(\tau_n(Z_m)\big)$ induces a map from $Y_{n+1}$ to $\tau(Y_n)$.
Unravelling definitions we see that it suffices to show that the following diagram commutes:
$$\xymatrix{
\tau_{m+n+1} X\ar[d]_{\hat\Delta_{(m,n+1)}^X}\ar[rr]^{\hat\Delta_{(m+n,1)}} && \tau\tau_{m+n} X\ar[d]^{\tau\big(\hat\Delta_{(m,n)}^X\big)} \\
\tau_{n+1}\tau_m X\ar[rr]^{\hat\Delta_{(n,1)}^{\tau_m X}} && \tau\tau_n\tau_m X
}$$
Reading the above diagram at the level of rings we see that it is a case of the associativity of $\Delta$ (cf. the commuting diagram in part (b) of Definition~\ref{iterative-hasse-system}).
Therefore, $\underline Y$ so defined is a dominant Hasse-Schmidt subvariety of $Z_m$.

Next we need to show that $\nabla_m\big(\underline Z(k)\big)=\underline Y(k)$.
First fix $p\in\underline Z(k)$.
Then $\nabla_m(p)\in Z_m\subseteq\tau_m X$ and so $\nabla_n\big(\nabla_m(p)\big)\in\tau_n\tau_m X$ for all $n$.
But $\nabla_n\big(\nabla_m(p)\big)=\hat\Delta_{(m,n)}\big(\nabla_{m+n}(p)\big)$.
Since $\nabla_{m+n}(p)\in Z_{m+n}$, $\nabla_n\big(\nabla_m(p)\big)\in Y_n$ for all $n$.
Hence $\nabla_m(p)\in\underline Y(k)$.
We have shown that $\nabla_m\big(\underline Z(k)\big)\subseteq\underline Y(k)$.

So far we have not used the assumption that $\underline{Z}$ has many $k$-rational points.
One consequence of this assumption is that $Y_n$ is the Zariski closure of $\nabla_n\big(\nabla_m\underline Z(k)\big)$, for all $n$.
Indeed, $Y_n$ is the image of $Z_{m+n}$ under the closed embedding $\hat\Delta_{(m,n)}$, $\nabla_n\big(\nabla_m(\underline Z(k))\big)$ is the image of $\nabla_{m+n}\big(\underline Z(k)\big)$ under the same map, and $\nabla_{m+n}\big(\underline Z(k)\big)$ is Zariski-dense in $Z_{m+n}$  by assumption.

It remains to show that if $q\in\underline Y(k)$ then $q\in\nabla_m\big(\underline Z(k)\big)$.
First note that it suffices to show that $q\in\nabla_m\big(X(k)\big)$.
Indeed, if $q=\nabla_m(p)$ then $\hat\Delta_{(m,n)}\big(\nabla_{m+n}(p)\big)=\nabla_n(q)\in Y_n$ for all $n$, so that $\nabla_{m+n}(p)\in Z_{m+n}$ for all $n$, which implies that $p\in\underline Z(k)$.
So we need to find $p\in X(k)$ such that $\nabla_m(p)=q$.
This will follow from the following claim

\begin{claim}
\label{nmclaim}
If $q\in \tau_m X(k)$ has the property that $\nabla_m(q)$ is contained in the Zariski closure of $\nabla_m\big(\nabla_m(X(k))\big)$, then $q=\nabla_m(p)$ for some $p\in X(k)$.
\end{claim}
\begin{proof}
Consider the commuting diagram
$$
\xymatrix{
\tau_m\tau_m X\ar[rr]^{\tau_m(\hat\pi^X_{m,0})}\ar[d]_{\hat\pi^{\tau_m X}_{m,0}} && \tau_m X\ar[d]^{\hat\pi^X_{m,0}}\\
\tau_m X\ar[rr]^{\hat\pi^X_{m,0}}&&X
}
$$
By the functoriality of $\nabla$ (cf. Proposition~4.7(a) of~\cite{paperA}) we have that
$\tau_m(\hat\pi^X_{m,0})\big(\nabla_m(q)\big)=\nabla_m\big(\hat \pi^X_{m,0}(q)\big)$.
So $q\in \nabla_m\big(X(k)\big)$ if and only if $\tau_m(\hat\pi^X_{m,0})\big(\nabla_m(q)\big)=q$.
But the latter identity says that $\nabla_m(q)$ satisfies a certain Zariski closed condition on $\tau_m(\tau_m X)$, namely the condition
$$\tau_m(\hat\pi^X_{m,0})(u)=\hat\pi^{\tau_m X}_{m,0}(u).$$
Since this condition is satisfied by all $u\in\nabla_m\big(\nabla_m(X(k))\big)$, and since $\nabla_m(q)$ is in the Zariski closure of $\nabla_m\big(\nabla_m(X(k))\big)$, we get that $q\in\nabla_m\big(X(k)\big)$, as desired.
\end{proof}

Now fix $q\in\underline Y(k)$.
So $\nabla_m(q)\in Y_m$, and the latter is in the Zariski closure of $\nabla_m\big(\nabla_m(X(k))\big)$ -- as it is the Zariski closure of $\nabla_m\big(\nabla_m(\underline Z(k))\big)$.
So by the claim, $\nabla_m(p)=q$ for some $p\in X(k)$, as desired.
This completes the proof of Proposition~\ref{nablanz}.
\end{proof}

\section{Generalised Hasse-Schmidt jet spaces}
\label{sec-hassejet}

\noindent
We intend to define a ``jet space'' associated to a point in a Hasse-Schmidt subscheme; it will be a linear Hasse-Schmidt subscheme of the jet space of the ambient algebraic variety at that point.
In the differential case, for finite-dimensional subvarieties, this was done by Pillay and Ziegler~\cite{pillayziegler03}, but their construction does not extend to infinite-dimensional differential varieties.
Staying with the differential setting for the moment, one might consider imitating the algebraic construction by defining the $n$th differential jet space of a differential variety at a point $p$ as the ``differential dual'' to the maximal differential ideal at $p$ modulo the $(n+1)$st power of that ideal.
This approach fails however, first because such a space is not naturally represented by a definable set in the langauge of differential rings, but also because they can be too small: they may not determine the differential variety.
This latter difficulty stems from non-noetherianity, or more specifically from the fact that there exist differential varieties with points that have the the property that the intersection of all the powers of the maximal ideal at the point is not trivial.\footnote{Phyllis Cassidy communicated to us the example of $x\delta^2x-\delta x=0$ at $x=0$.}
So neither the algebraic construction, nor the finite-dimensional differential construction of Pillay-Ziegler suggest extensions.
Our approach is to take the algebraic jet spaces of the sequence of algebraic schemes that define the Hasse-Schmidt subscheme, and then use this sequence to define a Hasse-Schmidt jet space.
In order to do so we make essential use of the interpolating map, which we discussed in Section~\ref{subsect-interp}, and which was developed in~\cite{paperA}.

Fix an iterative Hasse-Schmidt system $\underline\D$, an iterative $\underline \D$-ring $(k,E)$, a scheme $X$ of finite-type over $k$, a Hasse-Schmidt subscheme  $\underline Z=(Z_n)$ of $X$ over $k$, and a natural number $m$.
For each $n\in\mathbb{N}$, note that $\jet^mZ_n$ is a closed subscheme of $\jet^m\tau_n X$, and hence we can consider its scheme-theoretic image in $\tau_n\big(\jet^mX\big)$ under the interpolating map, which we denote by $T_n:=\overline{\phi_{m,n}^X(\jet^mZ_n)}$.
Recall that the scheme-theoretic image is just the smallest (with respect to closed embeddings) closed subscheme of the target through which the morphism factors.
Our ``overline'' notation is justified by the fact that when dealing with reduced schemes the scheme-theoretic image coincides with the induced reduced closed subscheme structure on the topological closure of the set-theoretic image.

\begin{lemma}
\label{djetisdvariety}
$\underline T:=\big(T_n:=\overline{\phi_{m,n}^X(\jet^mZ_n)}:n\in\mathbb{N}\big)$ is a Hasse-Schmidt subscheme of $\jet^mX$.
\end{lemma}

\begin{proof}
By functoriality, $\jet^m(\hat\pi_{n+1,n}):\jet^m\big(\tau_{n+1} X\big)\to\jet^m(\tau_n X)$ restricts to a map $\jet^m(Z_{n+1})\to\jet^mZ_n$.
Transforming this by the interpolating map (cf. part~(a) of Proposition~\ref{phi-properties-hasse}) yields that $\hat \pi_{n+1,n}^{\jet^mX}:\tau_{n+1}\jet^mX\to\tau_n\jet^mX$ restricts to a map from $T_{n+1}$ to $T_n$.
This proves the first condition of being a Hasse-Schmidt subscheme.

It remains to prove that for all $n\in\mathbb{N}$, $\hat\Delta_{n,1}^{\jet^mX}:\tau_{n+1}\jet^m X\to \tau\tau_n\jet^m X$ restricts to a map from $T_{n+1}$ to $\tau(T_n)$.
Parts~(b) and~(c) of Corollary~\ref{phi-properties-hasse} together give us the following compatibility of the interpolating map with $\Delta$:
$$\hat\Delta_{n,1}^{\jet^mX}\circ\phi_{m,n+1}^X=\tau(\phi_{m,n}^X)\circ\phi_{m,1}^{\tau_n X}\circ\jet^m(\hat\Delta_{n,1}^X).$$
Hence, to see where $\hat\Delta_{n,1}^{\jet^mX}$ takes $T_{n+1}=\overline{\phi_{m,n+1}^X(\jet^mZ_{n+1})}$, we can apply the right-hand-side of the above equality to $\jet^m Z_{n+1}$.
We have
$$\jet^m(\hat\Delta_{n,1}^X):\jet^m Z_{n+1}\to\jet^m\tau Z_n.$$
By functoriality of the interpolating map (Proposition~6.4(a) of~\cite{paperA}),
$$\phi^{\tau_n X}_{m,1}:\jet^m\tau Z_n\to\tau \jet^m Z_n.$$
Finally, since $\phi^X_{m,n}:\jet^m Z_n\to T_n$,
$$\tau(\phi^X_{m,n}):\tau\jet^m Z_n\to\tau T_n.$$
Hence,
$\hat\Delta_{n,1}^{\jet^mX}:T_{n+1}\to\tau T_n$, as desired.
\end{proof}

\begin{definition}[Hasse-Schmidt jet space]
\label{djetspace}
Suppose $\underline Z$ is a Hasse-Schmidt subscheme of $X$.
The {\em $m$th Hasse-Schmidt jet space} (or {\em $\underline\D$-jet space}) {\em of $\underline Z$} is the Hasse-Schmidt subscheme of $\jet^mX$ given by Lemma~\ref{djetisdvariety}.
That is,
$$\jet^m_{\underline\D}(\underline Z):=\big(\overline{\phi_{m,n}^X(\jet^mZ_n)}:n\in\mathbb{N}\big).$$
Given $a\in \underline{Z}(k)$ the {\em $m$th Hasse-Schmidt jet space of $\underline{Z}$ at $a$} is the Hasse-Schmidt subscheme of $\jet^m(X)_a$ given by
$$\jet^m_{\underline\D}(\underline Z)_a:=\big(\overline{\phi_{m,n}^X(\jet^mZ_n)}_{\nabla_n(a)}:n\in\mathbb{N}\big).$$
\end{definition}

Let us point out that for $a\in\underline{Z}(k)$, {\em $\jet^m_{\underline\D}(\underline Z)_a$ is indeed a Hasse-Schmidt subscheme of $\jet^m(X)_a$}.
As before, let $T_n:=\overline{\phi_{m,n}^X(\jet^mZ_n)}$
and let $(T_n)_{\nabla_n(a)}$ be the fibre over $\nabla_n(a)$ of $\tau_n\jet^mX\to \tau_nX$ restricted to $T_n$.
So $\jet^m_{\underline\D}(\underline Z)_a$ is given by the sequence $\big((T_n)_{\nabla_n(a)}:n\in\mathbb{N}\big)$.
First of all,
$(\tau_n\jet^mX)_{\nabla_n(a)}=\tau_n(\jet^m(X)_a)$
by Fact~\ref{nablafunctorialatpoints}, and so $(T_n)_{\nabla_n(a)}$ is a closed subscheme of the $n$th prolongation of $\jet^m(X)_a$.
Moreover, the structure morphism $\tau_{n+1}(\jet^m(X)_a)\to\tau_n(\jet^m(X)_a)$ is just the restriction of $\hat\pi_{n+1,n}:\tau_{n+1}\jet^mX\to\tau_n\jet^mX$.
Hence, as we have already shown that $T_{n+1}$ is sent to $T_n$, it follows from the functoriality of $\hat\pi$ (this is Proposition~4.8(b) of~\cite{paperA}) that $(T_{n+1})_{\nabla_{n+1}(a)}$ is sent to $(T_n)_{\nabla_n(a)}$.
A similar argument shows that $\big((T_n)_{\nabla_n(a)}:n\in\mathbb{N}\big)$ satisfies the second condition of being a Hasse-Schmidt subscheme, namely that $(T_{n+1})_{\nabla_{n+1}(a)}$ is sent to $\tau\big((T_n)_{\nabla_n(a)}\big)$ under the iterativity map $\hat\Delta_{(n,1)}$.

\begin{remark}
\label{jetfibreprop}
Suppose $a\in\underline{Z}(k)$.
\begin{itemize}
\item[(a)]
{\em $\jet^m_{\underline\D}(\underline Z)_a(k)$ is the fibre of $\jet^m_{\underline\D}(\underline Z)(k)$ over~$a$.}
That is
$$ \jet^m_{\underline\D}(\underline Z)_a(k)=\big\{\lambda\in\jet^m(X)_a(k):(a,\lambda)\in\jet^m_{\underline\D}(\underline Z)(k)\big\}.$$
Indeed, $\nabla_n(a,\lambda)\in T_n(k)$ if and only if $\nabla_n(\lambda)\in (T_n)_{\nabla_n(a)}(k)$.
\item[(b)]
{\em $\jet^m_{\underline\D}(\underline Z)_a$ is an irreducible Hasse-Schmidt subvariety of $\jet^m(X)_a$.}
Indeed, note that as $\jet^mZ_n$ is a linear subspace of $(\jet^m\tau_n X)|_{Z_n}$ over $Z_n$, and $\phi_{m,n}^X$ is a morphism of linear spaces over $\tau_n X$, the scheme-theoretic image $\overline{\phi^X_{m,n}(\jet^mZ_n)}$ is a linear subspace of $(\tau_n\jet^mX)|_{Z_n}$ over $Z_n$.
It follows that even though $\jet^mZ_n$ and $\overline{\phi^X_{m,n}(\jet^mZ_n)}$ are not necessarily reduced or irreducible, their fibres above points in $Z_n$, being vector groups, are necessarily reduced and irreducible.
\item[(c)]
{\em If $k$ is a field and $a\in X(k)$ is smooth, then $\jet^m_{\underline\D}(X)_a=\jet^m(X)_a$.}
Or more precisely, $\jet^m_{\underline\D}(\underline{X})_a=\underline{\jet^m(X)_a}$.
Indeed, by definition 
$\jet^m_{\underline\D}(\underline X)_a=\big(\overline{\phi^X_{m,n}\big(\jet^m(\tau_n X)\big)}_{\nabla_n(a)} \ : \ n\in\mathbb{N}\big)$.
But
$$\overline{\phi^X_{m,n}\big(\jet^m\tau_n X\big)}_{\nabla_n(a)}=(\tau_n\jet^mX)_{\nabla_n(a)}=\tau_n\big(\jet^m(X)_a\big)$$
where the first equality is by the surjectivity of the interpolating map at smooth points of $X$ (Fact~\ref{gensurjective}) and the second is Fact~\ref{nablafunctorialatpoints}.
\end{itemize}
\end{remark}

\subsection{Main results}
We establish in this section some fundamental properties of Hasse-Schmidt jet spaces, and prove also that they can be used to linearise Hasse-Schmidt subvarieties.
For these results we specialise to the case of varieties over fields and fix the following data:
an iterative Hasse-Schmidt system $\underline{\mathcal D}$, an iterative $\underline{\mathcal D}$-field $(k,E)$, an algebraic variety $X$ over $k$, a Hasse-Schmidt subvariety $\underline Z=(Z_n)$ of $X$ over $k$, and $m\geq 1$.

\begin{lemma}
\label{simpleatgeneral}
For sufficiently general $a\in\underline Z(k)$, and all $n\geq 0$,
$\phi_{m,n}^X$ restricts to a surjective (linear) morphism from $\jet^m(Z_n)_{\nabla_n(a)}$ to $\overline{\phi^X_{m,n}(\jet^mZ_n)}_{\nabla_n(a)}$.
More precisely, there exists a sequence of dense Zariski-open subsets $U_r\subseteq Z_r$, such that the above holds if $\nabla_r(a)\in U_r(k)$ for all $r$.

In particular, for such $a\in\underline Z(k)$,
$$\jet^m_{\underline\D}(\underline Z)_a(k)=\big\{\lambda\in\jet^m(X)_a(k):\nabla_n(\lambda)\in\phi^X_{m,n}\big(\jet^m(Z_n)_{\nabla_n(a)}(k)\big),\forall n\in\mathbb N\big\}.$$
\end{lemma}

\begin{proof}
As discussed earlier, the set of $k^{\alg}$-points of $\overline{\phi^X_{m,n}(\jet^mZ_n)}$ is the Zariski-closure of $\phi^X_{m,n}\big(\jet^mZ_n(k^{\alg})\big)$.
As the latter is a constructible set, the $U_n$ can be chosen so as to ensure that
$\overline{\phi^X_{m,n}(\jet^mZ_n)}_{\nabla_n(a)}(k^{\alg})$ is the Zariski closure of $\phi^X_{m,n}\big(\jet^m(Z_n)_{\nabla_n(a)}(k^{\alg})\big)$.
But $\phi^X_{m,n}\big(\jet^m(Z_n)_{\nabla_n(a)}(k^{\alg})\big)$ is already Zariski-closed since $\phi^X_{m,n}$ is linear on $\jet^m\big(\tau_n  X\big)_{\nabla_n(a)}$ and $\jet^m(Z_n)_{\nabla_n(a)}$ is a linear subvariety of $\jet^m\big(\tau_n  X\big)_{\nabla_n(a)}$.
We thus have
$$\overline{\phi^X_{m,n}(\jet^mZ_n)}_{\nabla_n(a)}(k^{\alg})=\phi^X_{m,n}\big(\jet^m(Z_n)_{\nabla_n(a)}(k^{\alg})\big).$$
As these are reduced schemes, this means that $\phi_{m,n}^X$ restricts to a surjective morphism from $\jet^m(Z_n)_{\nabla_n(a)}$ to $\overline{\phi^X_{m,n}(\jet^mZ_n)}_{\nabla_n(a)}$, as desired.
Moreover, by linearity over $k$, $\phi_{m,n}^X$ restricts to a surjective map between their sets of $k$-rational points also, so that
$$\overline{\phi^X_{m,n}(\jet^mZ_n)}_{\nabla_n(a)}(k)=\phi^X_{m,n}\big(\jet^m(Z_n)_{\nabla_n(a)}(k)\big).$$
Now the ``in particular'' clause follows since
$$\jet^m_{\underline\D}(\underline Z)_a(k)=\big\{\lambda\in\jet^m(X)_a(k):\nabla_n(\lambda)\in\overline{\phi^X_{m,n}(\jet^mZ_n)}_{\nabla_n(a)}(k),\forall n\in\mathbb N\big\}$$
by definition
\end{proof}

It is convenient at this point to formalise what we mean by ``sufficiently general''.

\begin{definition}[Good locus]
\label{goodlocus}
By the {\em good locus} of $\underline Z(k)$ let us mean those points $a\in\underline Z(k)$ satisfying the following three properties:
\begin{itemize}
\item[(1)]
$a$ is smooth point of $X$,
\item[(2)]
for all $n\geq 0$, $\hat\pi^X_{n+1,n}$ restricted to $Z_{n+1}$ is smooth at $\nabla_{n+1}(a)$, and
\item[(3)]
the conclusion of Lemma~\ref{simpleatgeneral} holds.
\end{itemize}
\end{definition}

Our jet spaces are only (provably) well-behaved at points in the good locus.
Loosely speaking, at least in the well-behaved contexts that we are most interested in, a ``generic'' point on a ``generic'' Hasse-Schmidt subvariety of $X$ will be in the good locus.
To make this precise, note first of all that as long as $Z_0$ has nonempty intersection with the smooth locus of $X$, there exists a sequence of dense Zariski-open subsets $U_r\subseteq Z_r$ such that $a\in X(k)$ is in the good locus of $\underline Z(k)$ if and only if $\nabla_r(a)\in U_r(k)$ for all $r\geq 0$.
So, if in addition $(k,E)$ is rich (Definition~\ref{rich}), then the good locus is not empty.
Recall that if $\underline\D$ extends to fields then we can achieve richness by passing to an extension (Corollary~\ref{densepoints}).

We have not yet dealt with the question of when the Hasse-Schmidt jet space is a {\em dominant} Hasse-Schmidt subvariety.
Without dominance our jet spaces may not have enough rational points in any extension of $(k,E)$.
Unfortunately, Hasse-Schmidt jet spaces of dominant Hasse-Schmidt subvarieties need not themselves be dominant, as the following example demonstrates.

\begin{example}
\label{nondomdifferential}
We consider (ordinary) iterative Hasse-Schmidt differential fields.
That is, we are working in the Hasse-Schmidt system $\operatorname{HSD}_1$ and we have an iterative Hasse-Schmidt differential field $\big(k,{\bf D}=(D_0,D_1,\dots)\big)$.
(See Example~\ref{differential} and Proposition~\ref{iterativedifferential}).
Suppose $\operatorname{char}(k)=p>0$ and consider the Hasse-Schmidt subvariety of the affine line defined by $\big(D_1(x)\big)^p=x$.
That is, let $\underline Z$ be the dominant Hasse-Schmidt subvariety of $\mathbb{A}^1$ obtained by applying Lemma~\ref{makehasse} to the sequence $(Y_n)$ where $Y_1$ is given by $y^p=x$ in $\tau_1(\mathbb{A}^1)=\spec(k[x,y])$, and $Y_i=\tau_i(\mathbb{A}^1)$ for all $i\neq 1$.
Then $Z_0=Y_0=\spec(k[x])$ and $Z_1=Y_1=\spec\big(k[x,y]/(y^p-x)\big)$.
Note that $Z_1\to Z_0$ is inseparable.
Now, since the algebraic tangent space coincides with the first algebraic jet space, and since the interpolating map from the prolongation of a tangent space to the tangent space of a prolongation is the identity (i.e. prolongations commute with tangent spaces), we get that $\jet_\D^1(\underline Z)$ is given by the sequence of tangent spaces $(TZ_n:n<\omega)$.
A straightforward calculation shows that $TZ_1\to TZ_0$ is not dominant, and so $\jet_\D^1(\underline Z)$ is not a dominant Hasse-Schmidt subvariety of $T\mathbb{A}^1$.
In fact, for any nonzero $a\in k$, the Hasse-Schmidt jet space at $a$, $\jet_\D^1(\underline Z)_a$, is not a dominant Hasse-Schmidt subvariety of $(T\mathbb{A}^1)_a$.
This is ultimately due to the inseparability of the morphism $\hat\pi_{1,0}:Z_{1}\to Z_0$.
The following proposition explains that such inseparability is the only obstacle.
\end{example}

\begin{proposition}
\label{djet-properties}
Suppose $\underline Z$ is a dominant and separable $\underline\D$-subvariety over $k$.
Then $\jet^m_{\underline\D}(\underline Z)$ is dominant.
Moreover, for $a\in\underline{Z}(k)$ in the good locus, $\jet^m_{\underline \D}(\underline Z)_a$ is dominant.
\end{proposition}

\begin{proof}
As before let $T_n:=\overline{\phi_{m,n}^X(\jet^mZ_n)}$ be the scheme-theoretic image of $\jet^mZ_n$ under $\phi_{m,n}^X$.
The commuting diagram in part~(a) of Corollary~\ref{phi-properties-hasse} restricts to
$$\xymatrix{
\jet^mZ_{n+1}\ar[d]^{\phi^X_{m,n+1}}\ar[rr]^{\jet^m(\hat\pi^X_{n+1,n})} & & \jet^mZ_n\ar[d]^{\phi^X_{m,n}}\\
T_{n+1}\ar[rr]^{\hat\pi_{n+1,n}^{\jet^mX}} && T_n
}$$
Now $\hat\pi_{n+1,n}:Z_{n+1}\to Z_n$ is dominant and separable by assumption.
It follows that $\jet^m(\hat\pi_{n+1,n}):\jet^mZ_{n+1}\to\jet^mZ_n$ is dominant (cf. Lemma~5.9 of~\cite{paperA}).
As the two vertical arrows are also dominant by definition, so is $T_{n+1}\to T_n$, as desired.

For the ``moreover'' clause, we restrict the above diagram to $a\in\underline Z(k)$:
$$\xymatrix{
\jet^m(Z_{n+1})_{\nabla_{n+1}(a)}\ar[d]^{\phi^X_{m,n+1}}\ar[rr]^{\ \ \jet^m(\hat\pi^X_{n+1,n})} & & \jet^m(Z_n)_{\nabla_n(a)}\ar[d]^{\phi^X_{m,n}}\\
(T_{n+1})_{\nabla_{n+1}(a)}\ar[rr]^{\hat\pi_{n+1,n}^{\jet^mX}} && (T_n)_{\nabla_n(a)}
}$$
Now, the separability and dominance of $\hat\pi_{n+1,n}:Z_{n+1}\to Z_n$ imply not only the dominance of $\jet^m(\hat\pi_{n+1,n}):\jet^mZ_{n+1}\to\jet^mZ_n$, but also the surjectivity of that map restricted to fibres above smooth points of $\hat\pi_{n+1,n}$ (see for example the proof of Lemma~5.9 of~\cite{paperA}).
This smoothness is also guaranteed by being in the good locus (condition~(2) of Definition~\ref{goodlocus}).
So the top horizontal arrow is surjective.
On the other hand, Lemma~\ref{simpleatgeneral} tells us that the two vertical arrows are also surjective.
Hence $(T_{n+1})_{\nabla_{n+1}(a)}\to (T_{n})_{\nabla_n(a)}$ is surjective for all $n$, implying that $\jet^m_{\underline\D}(\underline Z)_a$ is a dominant Hasse-Schmidt subvariety.
\end{proof}

The following theorem says that the Hasse-Schmidt jet spaces do indeed accomplish the task for which they were constructed.

\begin{theorem}
\label{maintheorem1}
Hasse-Schmidt subvarieties are determined by their jets:
Suppose
\begin{itemize}
\item
$(k,E)$ is a rich iterative $\underline\cD$-field,
\item
$X$ is an algebraic variety over $k$,
\item
$\underline Z$ and $\underline Z'$ are irreducible separable dominant $\underline\cD$-subvarieties of $X$ over $k$,
\item
and $a\in\underline Z(k)\cap\underline Z'(k)$ is in the good locus of both $\underline Z$ and $\underline Z'$.
\end{itemize}
If $\jet^m_{\underline\D}(\nabla_r\underline Z)_{\nabla_r(a)}(k)= \jet^m_{\underline\D}(\nabla_r\underline Z')_{\nabla_r(a)}(k)$ for all $m\geq 1$ and $r\geq 0$,
then $\underline Z=\underline Z'$.
\end{theorem}

Note that if $\underline\cD$ extends to fields, as it does in the differential  and difference cases (cf. Corollary~\ref{extendtofielddifferential} and Proposition~\ref{difference-example}), then by Corollary~\ref{densepoints} the hypothesis of richness in the above theorem can always be satisfied by passing to an extension.

\begin{proof}[Proof of Theorem~\ref{maintheorem1}]
Recall that $\nabla_r\underline Z$ is a dominant Hasse-Schmidt subvariety of $\tau_rX$ whose $k$-points are $\nabla_r\big(\underline Z(k)\big)$.
Its existence and dominance is ensured by Proposition~\ref{nablanz}
(this already requires richness of $(k,E)$ and dominance of $\underline Z$),
and by construction it's defining sequence is $\big(\hat\Delta_{(r,n)}(Z_{r+n})\subseteq\tau_n\tau_rX :n\in\mathbb N\big)$.
These subvarieties, together with the morphisms obtained by restricting $\hat\pi_{n+1,n}^{\tau_rX}$, are isomorphic to $(Z_{r+n}:n\in\mathbb N)$ with the restrictions of $\hat\pi^X_{r+n+1,r+n}$.
Hence $\nabla_r\underline Z$ is also separable, and $\nabla_r(a)$ is also in the good locus of $\nabla_r\underline Z$.

Fixing $r$ and taking jet spaces at $\nabla_r(a)$,
we have that $\jet^m_{\underline\D}(\nabla_r\underline Z)_{\nabla_r(a)}$ is a dominant Hasse-Schmidt subvariety by Proposition~\ref{djet-properties}.
We have already pointed out, in Remark~\ref{jetfibreprop}(b), that it is irreducible.
Hence, $k$ being rich, the $k$-points of $\jet^m_{\underline\D}(\nabla_r\underline Z)_{\nabla_r(a)}$ are Zariski-dense in
\begin{eqnarray*}
\big[\jet^m_{\underline\D}(\nabla_r\underline Z)_{\nabla_r(a)}\big]_0
&=&
\jet^m\big((\nabla_r\underline Z)_0\big)_{\nabla_r(a)}\\
&=&
\jet^m(Z_r)_{\nabla_r(a)}
\end{eqnarray*}
where the first equality is by definition of the Hasse-Schmidt jet space (as $\phi_{m,0}^{\tau_r X}=\id$).
So $\jet^m_{\underline\D}(\nabla_r\underline Z)_{\nabla_r(a)}(k)$ is Zariski-dense in $\jet^m(Z_r)_{\nabla_r(a)}$.
Similarly for $\underline Z'$.
Since $\jet^m_{\underline\D}(\nabla_r\underline Z)_{\nabla_r(a)}(k)= \jet^m_{\underline\D}(\nabla_r\underline Z')_{\nabla_r(a)}(k)$ by assumption, taking Zariski-closures, we get
$\jet^m(Z_r)_{\nabla_r(a)}= \jet^m(Z'_r)_{\nabla_r(a)}$,
 for all $m\in\mathbb{N}$.
Since the (algebraic) jet spaces of an irreducible algebraic subvariety at a point determine that subvariety, $Z_r=Z_r'$ for all $r\in\mathbb{N}$.
Hence $\underline Z=\underline Z'$.
\end{proof}

\begin{question}
Can the assumption that $a$ lies in the good locus be dropped?
\end{question}

\subsection{Canonical bases in differentially closed fields}
\label{differentialsection}
Let us specialise to the differential context and extract the model-theoretic content of the above theorem.
For this section we assume familiarity with basic notions and notations from stability theory such as those of canonical bases and definable closure.

First recall the setting.
We work with the generalised Hasse-Schmidt system $\hsd_e$ of Example~\ref{differential}, which extends to fields by Corollary~\ref{extendtofielddifferential}.
As a model-theoretic structure an iterative $\operatorname{HSD}_e$-field is just a field equipped with $e$ commuting iterative Hasse-Schmidt derivations (Proposition~\ref{iterativedifferential}).
Note that if $K$ is such a field, then the Kolchin closed subsets of $K^n$, or rather the countable intersections of such, are exactly the sets of the form $\underline Z(K)$ where $\underline Z$ is a dominant affine $\hsd_e$-variety (see Example~\ref{makehassedifferential}).
The class of existentially closed iterative Hasse-Schmidt differential fields of characteristic $p$ is elementary, and is axiomatised by the complete theory $\operatorname{SCH}_{p,e}$ (see~\cite{ziegler03}).
We allow the possibility of $p=0$, in which case this is just the theory of differentially closed fields of characteristic zero in $e$ commuting (usual) derivations.
It is not hard to see that if $K\models\operatorname{SCH}_{p,e}$ is an $\aleph_1$-saturated model then it is rich, and so dominant $\hsd_e$-subvarieties are uniquely determined by their $K$-points.
If $k\subseteq K$ is an $\operatorname{HSD}_e$-subfield and $a\in K^n$, then the {\em $\hsd_e$-locus of $a$ over $k$} is the irreducible dominant $\hsd_e$-subvariety $\underline Z=(Z_r)$ of $\mathbb A_K^n$ where $Z_r$ is the Zariski-locus of $\nabla_r(a)$ over $k$.
By quantifier elimination, $\tp(a/k)$ just says that $a$ is in $\underline Z(K)$ but not in any proper Kolchin closed subset over $k$.
If $k$ is relatively algebraically closed in $K$ then this type is stationary, and the stability-theoretic {\em canonical base} of $\tp(a/k)$ is nothing other than the $\hsd_e$-subfield of $k$ generated by the minimal fields of definition of all the $Z_r$.

\begin{corollary}
\label{differentialmain}
Suppose $K\models\operatorname{SCH}_{p,e}$ is saturated, $k\subseteq K$ is a relatively algebraically closed $\operatorname{HSD}_e$-subfield of cardinality less than $|K|$, $a\in K^n$, and $\underline Z=\hsd_e\operatorname{-locus}(a/k)\subseteq \mathbb A_K^n$.
If $\underline Z$ is separable then 
$$\cb(a/k)\subseteq\dcl\left(\{a\}\cup\bigcup_{m\geq 1,r\geq 0}\jet_{\operatorname{HSD}_e}^m(\nabla_r\underline Z)_{\nabla_r(a)}(K)\right)$$
When $p=0$ only one jet space is required: there exist $m\geq 1$ and $r\geq 0$ such that  $\cb(a/k)\subseteq\dcl\big(a,\jet_{\operatorname{HSD}_e}^m(\nabla_r\underline Z)_{\nabla_r(a)}(K)\big)$.
\end{corollary}

\begin{proof}
We work in the rich iterative $\hsd_e$-field $K$.
Note first of all that $a$ is in the good locus of $\underline Z(K)$;
indeed, being in the good locus is a $k$-definable Zariski-dense open condition on each $Z_r$, a condition which is therefore met by the $k$-generic point $\nabla_r(a)\in Z_r$.
Note also that automorphisms of $K$ (as an iterative Hasse-Schmidt differential field) act on $\hsd_e$-subvarieties by acting on their $K$-points, and they preserve irreducibility, separability, dominance, and the good locus.
It follows that, if $\underline Z'$ is a conjugate of $\underline Z$ over $a$, that is if $\underline Z':=\sigma(\underline Z)$ where $\sigma$ is an automorphism of $K$ fixing $a$, then $\underline Z'$ is also
an irreducible separable dominant $\hsd_e$-subvariety of $\mathbb A_K^n$ with $a$ in its good locus.
We may therefore apply Theorem~\ref{maintheorem1}.
That is, if $\underline Z'$ is any conjugate of $\underline Z$ over $a$ with $\jet_{\hsd_e}^m(\nabla_r\underline Z)_{\nabla_r(a)}(K)=\jet_{\hsd_e}^m(\nabla_r\underline Z')_{\nabla_r(a)}(K)$, for all $m\geq 1$ and $r\geq 0$, then $\underline Z=\underline Z'$.
So all automorphisms fixing $a$ and each $\jet_{\hsd_e}^m(\nabla_r\underline Z)_{\nabla_r(a)}(K)$ also fix $\underline Z$, and hence, by the discussion preceding the statement of the corollary, must fix $\cb(a/k)$.
It follows (using saturation and stability) that $\cb(a/k)\subseteq\dcl\left(\{a\}\cup\bigcup_{m\geq 1,r\geq 0}\jet_{\operatorname{HSD}_e}^m(\nabla_r\underline Z)_{\nabla_r(a)}(K)\right)$.

When the characteristic is zero the canonical base is in fact a finitely generated differential field (by $\omega$-stability), and so only finitely many jet spaces are needed.
But then by choosing $m$ and $r$ sufficiently large, we get $\cb(a/k)\subseteq\dcl\big(a,\jet_{\operatorname{HSD}_e}^m(\nabla_r\underline Z)_{\nabla_r(a)}(K)\big)$ in that case.
\end{proof}

The above corollary generalises to possibly infinite-rank types the main results of Pillay and Ziegler on canonical bases of finite-rank types in $\operatorname{SCH}_{p,e}$ (see Theorem~1.1 and Proposition~6.3 of~\cite{pillayziegler03} for the characteristic zero and positive characteristic cases respectively).
Indeed, as we will see in the next section, if $\tp(a/k)$ is of finite-rank then our jet spaces $\jet_{\hsd_e}^m(\nabla_r\underline Z)_{\nabla_r(a)}(K)$ agree with the ones constructed by Pillay and Ziegler, and are thus finite-dimensional vector spaces over the constants of $K$.
Hence, in the finite-rank case, Corollary~\ref{differentialmain}  says that if $c=\cb(a/k)$ then $\tp(c/a)$ is internal to the constants.

\subsection{Hasse-Schmidt jets via $\mathcal D$-modules}
The differential jet spaces of finite-dimensional differential varieties constructed by Pillay and Ziegler~\cite{pillayziegler03} were given explicitly in terms of the $\delta$-module structure on the algebraic jet space of the ambient algebraic variety.
Their description uses the finite-dimensionality in an essential manner, and no exact analogue can be expected in our setting.
Nevertheless, it is possible to give a characterisation of the Hasse-Schmidt jets that is of a similar flavour, and that is the goal of this final subsection.
The characterisation, Theorem~\ref{existentialdmodulecriterion} below, also shows that for finite-dimensional Kolchin closed sets our jet spaces coincide with those of Pillay and Ziegler (at least at points in the good locus).
The use of the term ``$\underline\D$-modules'' in the title of this subsection is meant to be suggestive; we do not formally develop the theory of $\underline\D$-modules here.

Let us fix an iterative Hasse-Schmidt system $\underline\D$ over $A$, a rich $\underline\D$-field $(k,E)$, a variety $X$ over $k$, and a Hasse-Schmidt subvariety $\underline Z$ of $X$ over $k$.
Fix also a point $a\in\underline Z(k)$ and $m\in\mathbb{N}$.

For each $r\geq 0$ the morphism $\hat\pi_r:Z_{r+1}\to Z_r$ induces a $k$-linear map
$$\mathfrak{m}_{Z_r,\nabla_r(a)}/\mathfrak{m}_{Z_r,\nabla_r(a)}^{m+1}\to \mathfrak{m}_{Z_{r+1},\nabla_{r+1}(a)}/\mathfrak{m}_{Z_{r+1},\nabla_{r+1}(a)}^{m+1}.$$
Setting  $V_r:=\mathfrak{m}_{Z_r,\nabla_r(a)}/\mathfrak{m}_{Z_r,\nabla_r(a)}^{m+1}$ for brevity, we obtain a directed system
$$V_0\to V_1\to V_2\to\cdots$$
Taking $k$-duals we have a corresponding inverse system of restriction maps
$$V_0^*=\jet^m(Z_0)_a(k)\leftarrow V_1^*=\jet^m(Z_1)_{\nabla(a)}(k)\leftarrow V_2^*=\jet^m(Z_2)_{\nabla_2(a)}(k)\leftarrow\cdots$$

\begin{lemma}
\label{eonv}
For each $r\geq 0$, the canonical morphism $r^{Z_0}_r:\tau_r Z_0\times_k\D_r(k)\to Z_0$ induces an additive map $e_r:V_0\to V_r\otimes_k\D_r(k)$.
\end{lemma}

\begin{proof}
Since $Z_r\subseteq\tau_r Z_0$, $r^{Z_0}$ restricts to a morphism $Z_r\times_k\D_r(k)\to Z_0$.
On the other hand, we have
$$\xymatrix{
\spec(k)\ar[rr]^a && Z_0\\
\spec\big(\D_r^{E_r}(k)\big)\ar[rr]^{\nabla_r(a)\times_k\D_r(k)}\ar[u]^{\underline E_r}&&\tau_r Z_0\times_k\D_r(k)\ar[u]^{r^{Z_0}_r}
}$$
Indeed, $\nabla_r(a)$ is by definition the unique morphism that makes the above square commute.
So $r^{Z_0}_r$ maps $\nabla_r(a)\times_k\D_r(k)$ to $a$.
Hence it induces $V_0\to V_r\otimes_k\D_r(k)$.
\end{proof}

The maps $e_r:V_0\to V_r\otimes_k\D_r(k)$ endow $V_0$ with something resembling a ``Hasse-Schmidt module'' structure.
For example, while these maps are {\em not} $k$-linear they do satisfy
$$e_r(a\cdot\alpha)=E_r(a)\cdot e_r(\alpha)$$
for all $a\in k$ and $\alpha\in V_0$.

\begin{remark}
Note that in the case when $X$ is affine, $e_r$ is just the map induced by the homomorphism $E^{Z_0,0}_r:k[Z_0]\to k[\tau_r Z_0]\otimes_k\D_r(k)$ discussed in section~\ref{dalgebra}.
\end{remark}

\begin{proposition}
\label{modulejet}
Suppose $\lambda\in \jet^m(Z_0)_a(k)=V_0^*$ and $\gamma\in \jet^m(Z_r)_{\nabla_r(a)}(k)=V_r^*$.
Then the following are quivalent:
\begin{itemize}
\item[(i)]
$\nabla_r(a,\lambda)=\phi_{m,r}^X\big(\nabla_r(a),\gamma\big)$
\item[(ii)] The following diagram commutes
$$\xymatrix{
V_0\ar[rr]^\lambda\ar[d]_{e_r} && k\ar[d]^{E_r}\\
V_r\otimes_k\D_r(k)\ar[rr]^{\ \ \ \gamma\otimes_k\D_r(k)}&&\D_r(k)
}$$
\end{itemize}
\end{proposition}

\begin{proof}
This proof will require some further familiarity with the construction of the interpolating map in~\cite{paperA}.

Note that~(i) makes sense: $(a,\lambda)\in\jet^m Z_0(k)$ and
$$\big(\nabla_r(a),\gamma\big)\in\jet^m Z_r(k)\subseteq\jet^m\tau_r Z_0(k)\subseteq\jet^m\tau_r X(k)$$
so that both $\nabla_r(a,\lambda)$ and $\phi_{m,r}^X\big(\nabla_r(a),\gamma\big)$ lie in $\tau_r\jet^m Z_0(k)$.
In fact, under the usual identifications, they both live in
$\jet^m(Z_0)_{\widehat{\nabla_r(a)}}\big(\D_r^{E_r}(k)\big)$, where $\widehat{\nabla_r(a)}:\spec\big(\D_r^{E_r}(k)\big)\to Z_0$ is the $\D_r^{E_r}(k)$-point of $Z_0$ associated to $\nabla_r(a)\in\tau_r Z_0(k)$.

\begin{claim}
\label{z0jetatnr}
$\jet^m(Z_0)_{\widehat{\nabla_r(a)}}\big(\D_r^{E_r}(k)\big) = \hom_{\D_r^{E_r}(k)}\big(V_0\otimes_k\D_r^{E_r}(k),\D_r^{E_r}(k)\big)$
\end{claim}

\begin{proof}
We have
$\jet^m(Z_0)_{\widehat{\nabla_r(a)}}\big(\D_r^{E_r}(k)\big) = \hom_{\D_r^{E_r}(k)}\big(\widehat{\nabla_r(a)}^*(\mathcal{I}/\mathcal{I}^{m+1}) \ , \ \D_r^{E_r}(k)\big)$ where $\mathcal{I}$ is the kernel of the map $\O_{Z_0}\otimes_k\O_{Z_0}\to\O_{Z_0}$ given by $f\otimes g\mapsto fg$ (cf. section~5 of~\cite{paperA}).
On the other hand,
$$\xymatrix{
\spec(k)\ar[rr]^a && Z_0\\
&\spec\big(\D_r^{E_r}(k)\big)\ar[ul]^{\underline{E_r}}\ar[ur]_{\widehat{\nabla_r(a)}}
}$$
commutes.
So $\widehat{\nabla_r(a)}^*(\mathcal{I}/\mathcal{I}^{m+1})=\underline{E_r}^*a^*(\mathcal{I}/\mathcal{I}^{m+1})$.
But
$$a^*(\mathcal{I}/\mathcal{I}^{m+1})=a^{-1}(\mathcal{I}/\mathcal{I}^{m+1})\otimes_{\O_{Z_0,a}}k=\O_{Z_0,a}/\mathfrak{m}_{Z_0,a}^{m+1}=V_0.$$
Hence, $\widehat{\nabla_r(a)}^*(\mathcal{I}/\mathcal{I}^{m+1})$ is (the sheaf of $\D_r^{E_r}(k)$-modules) $V_0\otimes_k\D_r^{E_r}(k)$.
\end{proof}

\begin{claim}
\label{interpofgammarefined}
As an element of $\jet^mZ_0\big(\D_r^{E_r}(k)\big)$,
$\phi\big(\nabla_r(a),\gamma\big) =\big(\widehat{\nabla_r(a)},\alpha\big)$ where $\alpha:V_0\otimes_k\D_r^{E_r}(k)\to\D_r^{E_r}(k)$ is given by
$$\alpha=\big([\gamma\otimes_k\D_r(k)]\circ e_r\big)\otimes_k\id_{\D_r^{E_r}(k)}$$
\end{claim}

\begin{proof}
We can view $\gamma\otimes_k\D_r(k)$ as a ({\em not} $k$-linear) map from $V_r\times_k\D_r(k)$ to $\D_r^{E_r}(k)$.
Precomposing with (the {\em not} $k$-linear) $e_r:V_0\to V_r\times_k\D_r(k)$, we get a map $[\gamma\otimes_k\D_r(k)]\circ e_r:V_0\to\D_r^{E_r}(k)$.
This map {\em is} $k$-linear.
Indeed, one can check this by tracing through the map (using, for example,~(\ref{ersplit}) below).
So the claim makes sense;
$\big([\gamma\otimes_k\D_r(k)]\circ e_r\big)\otimes_k\id_{\D_r^{E_r}(k)}\big):V_0\otimes_k\D_r^{E_r}(k)\to \D_r^{E_r}(k)$ is a well-defined $\D_r^{E_r}(k)$-linear map.

To prove the claim we first describe $\phi\big(\nabla_r(a),\gamma\big)$ using Claim~\ref{z0jetatnr} and the construction of the interpolating map in~\cite{paperA}.
Applying $\jet^m$ functor  to $r_r^{Z_0}\upharpoonright Z_r\otimes_k\D_r(k)$ induces a map
$$v: \jet^m\big(Z_r\times_k\D_r(k)\big)_{\nabla_r(a)\times_k\D_r(k)}\big(\D_r(k)\big)\to \jet^m(Z_0)_{\widehat{\nabla_r(a)}}\big(\D_r^{E_r}(k)\big)$$
by Lemma~6.2 of~\cite{paperA}.
Since $\D_r^{E_r}(k)=\D_r(k)$ as rings, Claim~\ref{z0jetatnr} tells us
$$\jet^m(Z_0)_{\widehat{\nabla_r(a)}}\big(\D_r^{E_r}(k)\big) = \hom_{\D_r(k)}\big(V_0\otimes_k\D_r^{E_r}(k),\D_r(k)\big).$$
On the other hand
$$\jet^m\big(Z_r\times_k\D_r(k)\big)_{\nabla_r(a)\times_k\D_r(k)}\big(\D_r(k)\big)=\hom_{\D_r(k)}\big(V_r\otimes_k\D_r(k), \D_r(k)\big).$$
Hence $v$ is dual to a $\D_r(k)$-linear map
$f:V_0\otimes_k\D_r^{E_r}(k)\to V_r\otimes_k\D_r(k)$.
By definition of the interpolating map in section~6 of~\cite{paperA},
\begin{eqnarray}
\label{interpofgamma}
\phi\big(\nabla_r(a),\gamma\big)
&=&
\big(\widehat{\nabla_r(a)},[\gamma\otimes_k\D_r(k)]\circ f\big)
\end{eqnarray}
On the other hand, $f$ is induced by the Weil representing morphism $\tau_rZ_0\times_k\D_r(k)\to Z_0\times_k\D_r^{E_r}(k)$.
Since $e_r:V_0\to V_0\times_k\D_r(k)$ is induced by $r^{Z_0}_r$, which is the above morphism composed with the projection $X\times_k\D_r^{E_r}(k)\to X$, it follows that
\begin{eqnarray}
\label{ersplit}
e_r
&=&
f\circ(\id_{V_0},1_{\D_r^{E_r}(k)})
\end{eqnarray}
where $(\id_{V_0},1_{\D_r^{E_r}(k)}):V_0\to V_0\otimes_k\D_r^{E_r}(k)$.
It is then not hard to see that
$$\big([\gamma\otimes_k\D_r(k)]\circ e_r\big)\otimes_k\id_{\D_r^{E_r}(k)}\big)=\big(\gamma\otimes_k\D_r(k)\big)\circ f.$$
Claim~\ref{interpofgammarefined} now follows from~(\ref{interpofgamma}).
\end{proof}

\begin{claim}
\label{nablaoflambda}
As an element of $\jet^m Z_0\big(\D_r^{E_r}(k)\big)$,
$$\nabla_r(a,\lambda) =
\big(\widehat{\nabla_r(a)},(E_r\circ\lambda)\otimes_k\id_{\D_r^{E_r}(k)}\big).$$
\end{claim}

\begin{proof}
We are are viewing $\nabla_r(a,\lambda)$ as a $\D_r^{E_r}(k)$-point of $\jet^m Z_0$.
As such we have
$$\xymatrix{
\spec(k)\ar[rr]^{(a,\lambda)}&&\jet^m(Z_0)\\
&\spec\big(\D_r^{E_r}(k)\big)\ar[ul]^{\underline{E_r}}\ar[ur]_{\nabla_r(a,\lambda)}
}$$
Claim~\ref{nablaoflambda} follows.
\end{proof}

Finally, we have
\begin{eqnarray*}
\nabla_r(a,\lambda)=\phi\big(\nabla_r(a),\gamma\big) & \iff &
(E_r\circ\lambda)\otimes_k\id_{\D_r^{E_r}(k)} = \big([\gamma\otimes_k\D_r(k)]\circ e_r\big)\otimes_k\id_{\D_r^{E_r}(k)}\\
& \iff &
E_r\circ\lambda = [\gamma\otimes_k\D_r(k)]\circ e_r
\end{eqnarray*}
where the first equivalence is by Claims~\ref{interpofgammarefined} and~\ref{nablaoflambda}.
This completes the proof of Proposition~\ref{modulejet}.
\end{proof}

We now have a ``$\underline\D$-module" characterisation of the Hasse-Schmidt jet spaces.

\begin{theorem}
\label{existentialdmodulecriterion}
Suppose $\underline Z$ is a $\underline\D$-subvariety of an algebraic variety $X$ over a rich $\underline\D$-field $k$, and $a\in\underline Z(k)$ is in the good locus.
An algebraic jet $\lambda\in \jet^m(X)_a(k)$ is in $\jet^m_{\underline\D}(\underline{Z})_a(k)$ if and only if for all $r\geq 0$ there exists $\gamma_r\in\jet^m(Z_r)_{\nabla_r(a)}(k)$ extending $\lambda$, such that $E_r\circ\lambda=[\gamma_r\otimes_k\D_r(k)]\circ e_r$.
\end{theorem}

\begin{proof}
This is just the ``in particular'' clause of Lemma~\ref{simpleatgeneral} together with Proposition~\ref{modulejet} combined.
\end{proof}

Let us use this characterisation to see that our jet spaces for finite-dimensional Kolchin closed sets in the differential context agree with those of Pillay and Ziegler~\cite{pillayziegler03}.
For ease of exposition, we will focus on the ordinary differential setting in characteristic zero.
Suppose $(K,\delta)$ is a saturated model of $\operatorname{DCF}_0$, the theory of differentially closed fields in characteristic zero, viewed as a rich $\hsd_1$-field (see Subsection~\ref{differentialsection}).
Suppose $(X,s)$ is an {\em affine $D$-variety}; that is, $X\subseteq\mathbb A_K^n$ is an irreducible algebraic subvariety and $s=(s_1,\dots,s_n)$ is a tuple of regular functions on $X$ such that $\bar s=(\id,s):X\to\tau X$ is a regular section to $\hat\pi_{1,0}:\tau X\to X$.
Let $\underline Z\subseteq\mathbb A_K^n$ be the irreducible dominant $\hsd_1$-subvariety whose $K$-points form the finite-dimensional Kolchin closed set
$(X,s)^\sharp:=\{a\in X(K):\delta(a)=s(a)\}$.
We explain how $\jet^m_{\hsd_1}(\underline Z)(K)$ agrees with the ``$m$th differential jet space of $(X,s)^\sharp$ at $a$'' in the sense of Pillay and Ziegler (cf.~3.9 of~\cite{pillayziegler03}).

By the Zariski-denseness of $\underline Z(K)$ in $Z_0$ and of $\nabla_1\big(\underline Z(K)\big)$ in $Z_1$, it follows that $X=Z_0$ and $\bar s$ is an isomorphism between $X$ and $Z_1$.
Applying the algebraic jet functor at a fixed point $a\in\underline Z(K)$, we get an isomorphism
$\jet^m(\bar s)_a:\jet^m(X)_a\to\jet^m(Z_1)_{\nabla(a)}$.
Hence, if $\lambda\in V_0^*=\jet^m(X)_a(K)$ then $\jet^m(\bar s)_a(\lambda)$ is the unique element of $V_1^*=\jet^m(Z_1)_{\nabla(a)}(K)$ extending $\lambda$.
Moreover, because the differential equations here are of order $1$, it can be shown that the criterion given by Theorem~\ref{existentialdmodulecriterion} need only be checked for $r=1$.
That is, assuming $a$ is in the good locus of $\underline Z(K)$, $\lambda\in\jet^m_{\hsd_1}(\underline Z)_a(K)$ if and only if
\begin{eqnarray}
\label{hsdcond}
E_1\circ \lambda &=& \big[\jet^m(\bar s)_a(\lambda)\otimes_KK[\eta]/(\eta^2)\big]\circ e_1
\end{eqnarray}
as maps from $V_0$ to $K[\eta]/(\eta^2)$.
Note that by definition, $E_1\lambda(v)=\lambda(v)+\big(\delta\lambda(v)\big)\eta$ for all $v\in V_0$.
On the other hand, in section~3 of~\cite{pillayziegler03} Pillay and Ziegler explain how $s$ induces a $\delta$-module structure $\delta'$ on $V_0=\mathfrak m_{X,a}/\mathfrak m_{X,a}^{m+1}$.
It is not hard to verify that the map $\big[\jet^m(\bar s)_a(\lambda)\otimes_KK[\eta]/(\eta^2)\big]\circ e_1$ is given by $v\mapsto \lambda(v)+\big(\lambda\delta'(v)\big)\eta$.
Hence~(\ref{hsdcond}) is equivalent to $\delta\lambda(v)=\lambda\delta'(v)$ for all $v\in V_0$.
But this is exactly the defining criterion for the Pillay-Ziegler $m$th differential jet space of $(X,s)^\sharp$ at $a$.

\section{Appendix: Other Examples}

Throughout the main text of the paper we have carried along at least one motivating example, namely that of Hasse-Schmidt differential rings (cf.~\ref{differential}, \ref{prolongdifferential}, \ref{iterativedifferential}, \ref{makehassedifferential}, \ref{extendtofielddifferential}, \ref{nondomdifferential}, and~\ref{differentialmain}).
In this appendix we outline several other motivating examples.

\subsection{Rings with endomorphisms}
\label{difference}
Consider the Hasse-Schmidt system $\operatorname{End}=(\D_n \ | \ n\in\mathbb{N})$ where $\D_n$ is $\mathbb{S}^{n+1}$ with the product ring scheme structure, the $\mathbb{S}$-algebra structure given by the diagonal $s_n:\mathbb{S}\to\mathbb{S}^{n+1}$, and $\pi_{m,n}$ the natural co-ordinate projection.
Then an $\operatorname{End}$-ring $(k,E)$ is a ring $k$ together with a sequence of  endomorphisms $(\sigma_i:k\to k\ |\ i \in {\mathbb Z}_+)$, where $E_n:=(\id,\sigma_1,\sigma_2,\dots,\sigma_n)$.

A special case of this is when, for each $n> 0$, $\sigma_{2n}=\tau_1^n$ and $\sigma_{2n+1}=\tau_2^n$, where $\tau_1$ and $\tau_2$ are a pair of endomorphisms of $k$, possibly commuting, and possibly even satisfying the relation $\tau_2=\tau_1^{-1}$.
In this way one can make any {\em difference} ring -- a ring equipped with a distinguished automorphism -- into an $\operatorname{End}$-ring.

A rather more convenient Hasse-Schmidt system for dealing with rings equipped with $e$ commuting automorphisms would be to set $\D_n$ to be $\mathbb{S}^{(2n+1)^e}$ with $s_n$ still the diagonal embedding and $\pi_{n+1,n}$ the natural co-ordinate projection.
Then a ring $k$ with commuting automorphisms $\tau_1,\dots,\tau_e$ can be viewed as a $\operatorname{End}$-ring by setting
$$E_n(x)=\big(\tau_1^{\alpha_1}\tau_2^{\alpha_2}\cdots\tau_e^{\alpha_e} (x)\big)_{\{\alpha\in\mathbb{Z}^e:\text{ each }|\alpha_i|\leq n\}}$$
We can now impose an iterativity condition which will force the iterative $\operatorname{End}$-rings to be rings equipped with $e$ commuting automorphisms.
For ease of presentation, let us deviate slightly from standard multi-index notation and write $|\alpha|\leq n$ to mean that $|\alpha_i|\leq n$ for each $i=1,\dots, e$.
Then our iteration map, $\Delta_{(m,n)}:\D_{m+n} \to \D_{(m,n)}$, will be given by by
$(x_\alpha)_{|\alpha| \leq n+m} \mapsto \big((x_{\beta+\gamma})_{|\beta| \leq n}\big)_{|\gamma| \leq m}$.

\begin{proposition}
\label{difference-example}
The system $\Delta=(\Delta_{(m,n)}:m,n\in\mathbb{N})$, above, makes $\operatorname{End}$ into an iterative Hasse-Schmidt system.
Moreover, the $\Delta$-iterative $\operatorname{End}$-rings are exactly the rings equipped with $e$ commuting automorphisms.
Finally, the system $\operatorname{End}$ extends to fields.
\end{proposition}

\begin{proof}
We leave the straightforward (though somewhat notationally tedious) task of showing that $(\operatorname{End},\Delta)$ is an iterative system, to the reader.
If $(k,E)$ is an $\operatorname{End}$-ring then by the compatibility of $E$ with $\pi$ we can write  $E_n(x) = \big(\sigma_\alpha(x)\big)_{\{\alpha\in\mathbb{Z}^e:|\alpha|\leq n\}}$ where each $\sigma_\alpha$ is an endomorphisms of $k$.
Then for $(k,E)$ to be $\Delta$-iterative means exactly that
\begin{eqnarray}
\label{endit}
\sigma_\gamma\circ\sigma_\beta & =&\sigma_{\beta+\gamma} \ \ \text{for all $\beta,\gamma\in\mathbb{Z}^e$.}
\end{eqnarray}
Clearly, if $\sigma_\alpha=\tau_1^{\alpha_1}\tau_2^{\alpha_2}\cdots\tau_e^{\alpha_e}$ for all $\alpha\in\mathbb{Z}^e$, where $\tau_1,\dots,\tau_e$ are commuting automorphisms of $k$, then~(\ref{endit}) holds.
Conversely, for $i=1,\dots,e$, let $\tau_i:=\sigma_{(\dots,0,1,0,\dots)}$ where the $1$ is in the $i$th co-ordinate.
Then~(\ref{endit}) implies that the $\tau_1,\dots,\tau_e$ commute, are invertible, and
 $\sigma_\alpha=\tau_1^{\alpha_1}\tau_2^{\alpha_2}\cdots\tau_e^{\alpha_e}$ for all $\alpha\in\mathbb{Z}^e$.

To see that $\operatorname{End}$ extends to fields suppose $(R,E)$ is an iterative $\operatorname{End}$-integral domain and $K$ is the fraction field of $R$.
We need to extend each $E_n$ to a ring homomorphism $\tilde{E_n}:K\to\D_n(K)$.
It suffices to check that $E_n$ takes nonzero elements in $R$ to units in $\D_n(K)$.
But this is the case since the units in $\D_n(K)=K^{(2n+1)^e}$ are just those elements all of whose co-ordinates are nonzero, and $E_n(x)=\big(\tau_1^{\alpha_1}\tau_2^{\alpha_2}\cdots\tau_e^{\alpha_e} (x)\big)_{|\alpha|\leq n}$, where the $\tau_i$ are automorphisms of $R$.
\end{proof}

\begin{remark}
Note that it is not the case that $\operatorname{End}$-rings always localise, one must require that the
multiplicatively closed set by which one is localising is also closed under the
operators.
Note also that we really needed iterativity here in order to extend to fields:
if $R$ is an integral domain and $\sigma:R \to R$ is a nonconstant
endomorphism with a nontrivial kernel (\emph{eg} $R = {\mathbb Z}[x]$ and
$\sigma(f(x)) := f(0)$), then there is no extension of $\sigma$ to an
endomorphism of the field of fractions of $R$.
\end{remark}

Let us now make explicit the model-theoretic content of Theorem~\ref{maintheorem1} applied to this example.
We work in the theory $\operatorname{ACFA}$ of existentially closed difference fields (in one derivation), and in a saturated model $(K,\sigma)$ of this theory.
Analogously to the differential case discussed in subsection~\ref{differentialsection}, given a difference subfield $k\subseteq K$ and $a\in K^n$, we can define the {\em $\operatorname{End}$-locus of $a$ over $k$} to be the irreducible dominant $\operatorname{End}$-subvariety $\underline Z=(Z_r)$ of $\mathbb A_K^n$ where $Z_r$ is the Zariski-locus of $\nabla_r(a)=\big(a,\sigma(a),\dots,\sigma^r(a)\big)$ over $k$.
Because the theory does not admit quantifier elimination, the locus only captures the quantifier-free type of $a$ over $k$.
Nevertheless, if $k$ is algebraically closed then the (simplicity-theoretic) canonical base of $\tp(a/k)$ is an algebraic extension of the difference field generated by the minimal fields of definition of all the $Z_r$.
In the difference-field analogue of Corollary~\ref{differentialmain} we must therefore replace $\dcl$ by $\acl$, but otherwise the statement and the proof are the same:

\begin{corollary}
\label{differencemain}
Suppose $(K,\sigma)\models\operatorname{ACFA}$ is saturated, $k\subseteq K$ is an algebraically closed difference subfield of cardinality less than $|K|$, $a\in K^n$, and $\underline Z\subseteq \mathbb A_K^n$ is the $\operatorname{End}$-locus of $a$ over $k$.
If $\underline Z$ is separable then there exist $m\geq 1$ and $r\geq 0$ such that $\cb(a/k)\subseteq\acl\big(a,\jet_{\operatorname{End}}^m(\nabla_r\underline Z)_{\nabla_r(a)}(K)\big)$.
\end{corollary}

Analogously to the differential case, when $\tp(a/k)$ is of finite-rank our jet spaces agree with those of Pillay and Ziegler, and Corollary~\ref{differencemain} recovers Theorem~1.2 of~\cite{pillayziegler03}; namely, that $\cb(a/k)$ is almost internal to the fixed field of $(K,\sigma)$.

\subsection{Difference-differential rings}
\label{difference-differential}
We can combine the above example with the differential example.
A Hasse-Schmidt system that is convenient for the study of a ring equipped with one Hasse-Schmidt derivation together with an endomorphism might be the following:
$\displaystyle \D_n(R)=\prod_{i=0}^nR[\eta]/(\eta)^{n+1-i}$,
$s_n(r):=(r+(\eta)^{n+1-i}:i=0,\dots,n)$,
$\displaystyle \psi_n:\D_n(R)\to\prod_{i=0}^nR^{n+1-i}$ given by the standard monomial basis in each of the $n+1$ factors,
and $\pi_{m,n}:\D_m(R)\to\D_n(R)$ given by projecting onto the first $n$ coordinates and then taking the quotient $R[\eta]/(\eta)^{m+n+1-i}\to R[\eta](\eta)^{n+1-i}$ on each of the remaining factors.
Given a ring $k$ together with a Hasse-Schmidt derivation ${\bf D}$ and an endomorphism $\sigma$, we make $k$ into a $\underline\D$-ring by setting $E_n:k\to\D_n(k)$ to be the ring homomorphism
$$E_n(x)=\big(\sum_{j=0}^{n-i}\sigma^iD_j(x)\eta^j \ :i=0,1,\dots,n\big).$$

As before, if one wants to focus on the case of an automorphism a more convenient presentation  would be
$$\D_n(R)=\prod_{i=1}^nR[\eta]/(\eta)^{n+1-i} \ \times R[\eta]/(\eta)^{n+1} \times \prod_{i=1}^nR[\eta]/(\eta)^{n+1-i}$$
and
$$E_n(x)=\big(\sum_{j=0}^{n-i}\sigma^{-i}D_j(x)\eta^j, \sum_{j=0}^nD_j(x)\eta^j,\sum_{j=0}^{n-i}\sigma^iD_j(x)\eta^j :i=1,\dots,n\big).$$
We can then combine the iterativity maps for $\operatorname{HSD}$ and $\operatorname{End}$ to obtain
an iteration map $\Delta_{(m,n)}:\D_{m+n} \to \D_{(m,n)}$ given by
$$f_i(\eta)_{-{n+m} \leq i \leq n+m} \mapsto \big((f_{\alpha+\beta}(\zeta+\epsilon))_{-n \leq \alpha \leq n}\big)_{-m \leq \beta \leq m}.$$
The corresponding iterative Hasse-Schmidt rings are precisely rings equipped with an iterative Hasse-Schmidt derivation and an automorphism that commutes with the Hasse-Schmidt derivation.
Moreover, this iterative Hasse-Schmidt system will extend to fields.

\subsection{Higher $D$-rings}
\label{Dring}
As a final example we consider a higher order version of the $D$-rings studied by the second author in~\cite{Sc-thesis} and~\cite{scanlon2000}, see also Example~3.7 of~\cite{paperA}.
As we explain at the end of this section, higher $D$-rings specialise to both Hasse-Schmidt differential rings and to difference rings thought of as rings with difference operators.

Let $e$ be a positive integer and let $A := \ZZ[c_1,\ldots,c_e]$ be the polynomial ring in $e$ indeterminates.
We define a Hasse-Schmidt system over $A$ as follows.
For each $m \in \NN$, let
$$P_m(X,W) := \prod_{i=0}^{m-1} (X - i W) \in \ZZ[X,W]$$
where for convenience we set $P_0(X,W) := 1$.
 For $I \in \NN^e$ and $R$ an $A$-algebra define
$$\cD_I(R) := R[\epsilon_1,\ldots,\epsilon_e]/\big(P_{I_1+1}(\epsilon_1,c_1),\ldots,P_{I_e+1}(\epsilon_e,c_e)\big).$$
As $P_\ell(X,W)$ divides $P_m(X,W)$ for $\ell \leq m$, we have quotient maps $\pi_{I,J}:\cD_I(R) \to \cD_J(R)$ for $J \leq I$.
Since, $P_1(X,W) = X$, $\cD_{\boldsymbol{0}}(R) = R$.
As $P_m(X,c_\ell)$ is a monic polynomial over $k$, the rings $\cD_I(R)$ are free $R$-algebras with monomial basis $$\{ \epsilon^J : J \leq I \}.$$
So $\underline\cD=(\cD_I:I\in\mathbb{N}^e)$ is a Hasse-Schmidt system over $A$, albeit indexed by $\mathbb{N}^e$ and thus diverging slightly from our formalism.

Observe that the ring $\ZZ[W][X,Y]/(P_\ell(X,W),P_m(Y,W))$ is the coordinate ring of the reduced subscheme $X_{\ell,m}$ of ${\mathbb A}^2_{\ZZ[W]}$ whose underlying space is $\{ (i W, j W) : 0 \leq i < \ell, 0 \leq j < m \}$.
Visibly, $P_{\ell+m+1}(X+Y,W)$ is identically zero on $X_{\ell+1,m+1}$.
Hence, $$P_{\ell+m+1}(X+Y,W) \in (P_{\ell+1}(X,W),P_{m+1}(Y,W)).$$
This observation permits a definition of an iteration map.
Indeed, changing variables so as to separate out the roles of each of the applications of $\cD_I$, for $I$ and $J$ two multi-indices in $\NN^e$ and $R$ an $A$-algebra, let us write $\cD_I \circ \cD_J (R)$ as
$$R[X_1,\ldots,X_e,Y_1,\ldots,Y_e]/(P_{I_1}(X_1,c_1),\ldots,P_{I_e}(X_e,c_e),P_{I_1}(Y_1,c_1),\ldots,P_{I_e}(Y_e,c_e))$$
and
$$\cD_{I+J}(R) := R[Z_1,\ldots,Z_e]/(P_{I_1 + J_1}(Z_1,c_1),\ldots,P_{I_n+J_n}(Z_n,c_n)).$$
The iteration map
$\Delta_{I,J}:\cD_{I+J} \to \cD_I \circ \cD_J$
is then defined by
$Z_i \mapsto X_i + Y_i$
for $1 \leq i \leq n$.
Our observation that $P_{\ell+m+1}(X_i+Y_i,c_i)$ may be expressed as an $R$-linear combination of $P_{\ell+1}(X_i,c_i)$ and $P_{m+1}(Y_i,c_i)$ shows that $\Delta_{I,J}$ is a homomorphism of $R$-algebras.
Visibly these maps are associative and compatible with the projection maps defining the inverse system.

As usual, a $\underline\cD$-ring structure on an $A$-algebra $k$ is given by collection of $A$-algebra homomorphisms $E_I:k \to \cD_I(k)$ compatible with the identification $\cD_{\boldsymbol{0}}(k) = k$ and the maps $\pi_{I,J}:\cD_I(k) \to \cD_J(k)$ in the inverse system.
We may express each such map in terms of the monomial basis as  $$E_I(x) = \sum_{J \leq I} \partial_{I,J}(x) \epsilon^J.$$
However, it is not the case in general that for $J \leq I$ and $J \leq K$ that $\partial_{I,J} = \partial_{K,J}$.
For example, taking $e = 1$, we have  $\epsilon^2  = 0 \cdot \epsilon^0 + 0 \cdot \epsilon^1 + 1 \cdot \epsilon^2$ in $\cD_2(k)$ but $\epsilon^2 = 0 \cdot \epsilon^0 + e \cdot \epsilon^1$ in $\cD_1(k)$.
If we wish to express the $\underline\cD$-ring structure on $k$ via a single $\NN^e$-indexed sequence of operators $\delta_J:k \to k$, then instead of the monomial basis we should take $\{\beta_J:J\leq I\}$ as a basis for $\cD_I$, where
$$\beta_J(\epsilon_1,\dots,\epsilon_e) := \prod_{i=1}^e P_{J_i}(\epsilon_i,c_i).$$
Viewing the $\cD_I$ as finite free $\mathbb{S}$-algebras with respect to this basis, we have that if $(k,E)$ is a $\underline\cD$-ring then
$$ E_I(x)=\sum_{J\leq I}\partial_J(x)\beta_J$$
where $(\partial_J: J\in\mathbb{N}^e)$, are $A$-linear additive endomorphisms of $k$.

\begin{proposition}
\label{ithd}
Suppose $k$ is an $A$-algebra and $(\partial_I:I\in\mathbb{N }^e)$ is a set of $A$-linear additive endomorphisms of $k$.
For $i \leq e$, let $\sigma_i:=c_i\cdot\partial_i+ \id$ where $\partial_i:= \partial_{(0,\dots,0,1,0,\dots,0)}$ with the $1$ is in the $i$th co-ordinate.
For $K\in\mathbb{N}^e$, set $\sigma^K:=\sigma_1^{K_1}\circ\cdots\circ\sigma_e^{K_e}$.
Then setting $\displaystyle E_I(x)=\sum_{J\leq I}\partial_J(x)\beta_J$ for all $I\in\mathbb{N}^e$,
$(k,E)$ is an iterative $\underline\cD$-ring
if and only if the following two rules hold
\begin{itemize}
\item
Product rule:
$\displaystyle \partial_I(xy) = \sum_{J+K = I} \sigma^K \big(\partial_J (x)\big)\cdot\partial_K(y)$,
\item
Iteration rule:
$\displaystyle \partial_I\circ\partial_J=\binom{I+J}{I} \partial_{I+J}$.
\end{itemize}
\end{proposition}

To carry out this proof we need a few easy combinatorial lemmata. 
 
Let us start with a calculation allowing us to see the iteration rule.

\begin{lemma}
\label{sumexpression}
$P_\ell(X+Y,W) = \sum_{m=0}^{\ell} \binom{\ell}{m} P_m(X,W) P_{\ell - m}(Y,W)$
\end{lemma}
\begin{proof}
It suffices to show that the stated equality holds whenever one evaluates at points of the form $(aW,bW)$ where $a$ and $b$ are integers.  On the lefthand side, we have $P_\ell(aW + bW,W) = \prod_{i=0}^{\ell-1} ((a+b-i)W) = \ell ! \binom{a+b}{\ell} W^\ell$.   On the righthand side we have

\begin{eqnarray*}
\sum_{m=0}^\ell \binom{\ell}{m}P_m(aW,W) P_{\ell -m}(bW,W) & = & \sum_{m=0}^\ell \binom{\ell}{m} m! \binom{a}{m} W^m (\ell - m)! \binom{b}{\ell - m} W^{\ell - m}  \\ & = &  W^\ell \sum_{m=0}^\ell \frac{\ell! m! (\ell - m)!}{m! (\ell-m)!}  \binom{a}{m} \binom{b}{\ell - m} \\  & = &  W^\ell \ell! \sum_{m=0}^\ell \binom{a}{m} \binom{b}{\ell - m}  \\ & = & W^\ell \ell! \binom{a+b}{\ell}
\end{eqnarray*}
The last equality is obtained by comparing the coefficients of $W^\ell$ in the expansion of the equality $(1+W)^a (1 + W)^b = (1+W)^{a+b}$.
\end{proof}

Now
\begin{eqnarray*}
\Delta_{(m,n)}\circ E_{I+J}(x)
&=&
\sum_{K\leq I+J}\partial_K(x)\beta_K(X_1+Y_1,\dots,X_e+Y_e)\\
&=&
\sum_{K\leq I+J}\partial_K(x)\prod_{i=1}^e P_{K_i}(X_i+Y_i,c_i).
\end{eqnarray*}
Using Lemma~\ref{sumexpression} to expand this, one sees that $\Delta$-iterativity is equivalent to the iteration rule claimed by the proposition.

To see that the claimed Leibniz rule is equivalent to the $E_I$ being homomorphisms, we need to compute the product of two standard basis vectors.
First we observe:

\begin{lemma}
\label{productexpression}
$P_n(X,W) P_m(X,W) = \sum_{i=0}^n  i! \binom{n}{i} \binom{m}{i} W^i P_{m+n - i}(X)$
\end{lemma}
\begin{proof}
The case of $n=0$ is clear.  For the inductive step,
\begin{eqnarray*}
P_{n+1}  P_m  & = & \sum_{i=0}^n i! \binom{n}{i} \binom{m}{i} W^i  (X - (m+n-i)W + (m-i)W)P_{m+n-i} \\
  & = & \sum_{i=0}^n  i! \binom{n}{i} \binom{m}{i} W^i  P_{m+1+n-i} + i! \binom{n}{i} \binom{m}{i} W^{i+1} (m-i) P_{m+n-i} \\
  & = & \sum_{i=0}^{n+1} (i! \binom{n}{i} \binom{m}{i}  + (i-1)! \binom{n}{i-1} \binom{m}{i-1} (m-i+1)) W^iP_{m+n+1-i} \\
  & = & \sum_{i=0}^{n+1}  (\frac{m!}{(m-i)!} \binom{n}{i} + \binom{n}{i-1} \frac{(i-1)! m! (m-i+1)}{(m-i+1)! (i-1)!}) W^i P_{m+n+1-i} \\
  & = & \sum_{i=0}^{n+1} i! \binom{m}{i} \binom{n+1}{i} W^i P_{m+n+1-i}
\end{eqnarray*}
\end{proof}

Lemma~\ref{productexpression} leads to an expression for the product rule, but not the claimed one.  For the sake of definiteness, let us write down the Leibniz rule predicted by Lemma~\ref{productexpression}. Expanding the exponential in two different ways, we have 
\begin{eqnarray*} 
\sum \partial_L (ab)\beta_L & = & E(ab) \\ 
&=& E(a) E(b) \\
&=& \sum_{I,J} \partial_I(a) \partial_J(b)\beta_I\beta_J \\
&=& \sum_{I,J} \sum_K \partial_I(a) \partial_J(b) K! \binom{I}{K} \binom{J}{K} c^K\beta_{I+J-K}  
\end{eqnarray*}
So multiplicativity of $E$ amounts to the product rule:
\begin{eqnarray}
\label{hd-preprod}
\partial_L(ab) & = & \sum_{I+J = K+L}  K! c^K \binom{I}{K} \binom{J}{K} \partial_I(a) \partial_J(b).
\end{eqnarray}
To put~(\ref{hd-preprod}) in the form claimed by Proposition~\ref{ithd}, we should compute the iterates of $\sigma$.
Under the hypothesis of iterativity, if $\sigma(x) = c \partial_1(x) + x$, then $\sigma^n(x) = \sum_{i=0}^n c^i \frac{n!}{(n-i)!} \partial_i(x)$.    Indeed, $\sigma^n(x) = \sum_{i=0}^n c^i \binom{n}{i} \partial_1^i(x)$.  Via iterativity, we have $i! \partial_i = \partial_1^i$ so that $\binom{n}{i} \partial_1^i = \frac{n!}{(n-i)!} \partial_i$.
Putting together this observation with~(\ref{hd-preprod}), we compute
\begin{eqnarray*}
\partial_L(ab) & = & \sum_{I+J = K+L}  K! c^K \binom{I}{K} \binom{J}{K} \partial_I(a) \partial_J(b) \\
   & =& \sum_{I'+J=L}  \sum_{K=0}^J K! c^K \binom{I'+K}{K} \binom{J}{K} \partial_{I'+K}(a) \partial_J(b) \\ 
   &=& \sum_{I'+J=L} \sum_{K=0}^J K! c^K \binom{J}{K} \partial_K (\partial_{I'}(a)) \partial_J(b) \\
   &=& \sum_{I'+J=L} \sigma^J(\partial_{I'}(a)) \partial_J(b)
\end{eqnarray*} 
The computation is reversible, and so we get that~(\ref{hd-preprod}) is equivalent to the desired product rule.
This completes the proof of Proposition~\ref{ithd}.
\qed

\smallskip

Let us note some specializations.
If $A \to k$ factors through $\ZZ[c_1,\ldots,c_e] \to \ZZ[c_1,\ldots,c_e]/(c_1,\ldots,c_e) = \ZZ$, then an iterative $\underline\cD$-ring is simply an iterative Hasse-Schmidt differential ring.
If $k$ is a $\QQ$-algebra, then it follows from the iteration rule that $\displaystyle \partial_I = \frac{1}{I!} \partial_1^{I_1} \circ \cdots \circ \partial_e^{I_e}$ so that the full stack is already determined by the operators $\partial_1,\ldots,\partial_e$.
If $c_i$ is a unit in $R$, then  $\partial_i = c_i^{-1} (\sigma_i - \id)$.
Thus, in the case of $\QQ[c_1^{\pm 1},\ldots,c_e^{\pm 1}]$-algebras, the category of $\underline\cD$-algebras is equivalent to the that of difference algebras for $e$ commuting endomorphisms.
However, in positive characteristic, even when the parameters $c_i$ are units, it is not the case that a $\underline\cD$-ring is essentially just a difference ring.

Algebras over $\ZZ[c]$ with additive operators $D:R \to R$ satisfying $D(xy) = x D(y) + y D(x) + c D(x) D(y)$ were considered by the second author in~\cite{Sc-thesis} and~\cite{scanlon2000}.   Andr\'{e} developed a theory of confluence between difference and differential operators in~\cite{Andre} taking both operators $\sigma:R \to R$ and $\delta:R \to R$ as basic where $\sigma$ is a ring endomorphism and $\delta$ is an additive operator satisfying the twisted Leibniz rule $\delta(xy) = \sigma(x) \delta(y) + \delta(x) y$.  If there is some $b \in R$ with $\delta(b) \in R^\times$, then one may express $\sigma(x) = c \delta(x) + x$ where $c := \frac{\sigma(b) - b}{\delta(b)}$.  The operator $\delta$ is then a $D$-operator in the above sense.  

Hardouin develops a theory of iterative $q$-difference operators in~\cite{Har}.   Her axioms are very similar to ours (with $e=1$).  For instance, the Leibniz rules are exactly the same.  However, there are some major distinctions.  The parameter $c$ is $(q-1)t$ so that the operator $\delta_1(x) = \frac{\sigma_q(x) - x}{(q-1) t}$ where $\sigma_q:\CC(t) \to \CC(t)$ is the automorphism $f(t) \mapsto f(qt)$ is not $\ZZ[c]$-linear.  Additionally, her iteration rules involve the $q$-analogues of the binomial coefficients.  Most importantly, her exponential maps take values in noncommutative difference algebraic rings.   Some aspects of the $q$-iterative operators may be incorporated into our setting by working with the ring schemes $\cD_n(R) :=  R[\epsilon]/(\prod_{i=0}^{n-1} (\epsilon - q^i))$.
We leave to future work any further comparisons between these theories, as well as the more general issue of relaxing the definition of $\underline\cD$-rings so that the operators need not be linear over the base ring.


\end{document}